\documentclass[11pt,reqno]{amsart}

\usepackage{fullpage}
\usepackage{url}
\usepackage{bbm}
\usepackage{amssymb,amsmath}
\usepackage{mathrsfs,amsfonts,bm,amsthm,color,enumerate,stmaryrd}
\usepackage[all]{xy}
\definecolor{hot}{RGB}{65,105,225}
\usepackage[pagebackref=true,colorlinks=true, linkcolor=hot ,  citecolor=hot, urlcolor=hot]{hyperref}

\newtheorem{theorem}{Theorem}[section]
\newtheorem{lemma}[theorem]{Lemma}
\newtheorem{thrm}[theorem]{Theorem}

\newtheorem{conj}[theorem]{Conjecture}
\newtheorem{theorem-definition}[theorem]{Theorem-Definition}
\newtheorem{corollary}[theorem]{Corollary}
\newtheorem{cor}[theorem]{Corollary}
\newtheorem{proposition}[theorem]{Proposition}
\newtheorem{prop}[theorem]{Proposition}

\theoremstyle{definition}
\newtheorem{example}[theorem]{Example}
\newtheorem{definition}[theorem]{Definition}

\newtheorem{remark}[theorem]{Remark}
\newtheorem{rmk}[theorem]{Remark}

\numberwithin{figure}{section}

\def\bQ{\mathbb Q}
\def\bC{\mathbb C}
\def\bZ{\mathbb Z}
\def\divi{\mathrm{div}}
\def\cL{\mathcal L}
\def\be{\begin{equation}}
\def\ee{\end{equation}}
\def\bL{\mathbb L}
\def\bA{\mathbb A}
\def\cA{\mathcal A}
\def\bP{\mathbb P}
\def\ol{\overline}
\def\PV{\mathrm{PV}}
\def\cS{\mathcal S}
\def\Pic{\mathrm{Pic}}
\def\al{\alpha}
\def\bR{\mathbb R}
\def\cO{\mathcal O}
\def\Hom{\mathrm{Hom}}
\def\xra{\xrightarrow}

\def\cF{\mathcal F}
\def\cG{\mathcal G}
\def\cT{\mathcal T}
\def\ol{\overline}
\def\codim{\mathrm{codim}}

\author{Nero Budur}
\address{Department of Mathematics, KU Leuven, Celestijnenlaan 200B, 3001 Leuven, Belgium;  Yau Mathematical Sciences Center, Tsinghua University, 100084 Beijing, China;  Basque Center for Applied Mathematics, Mazarredo 14, 48009 Bilbao, Spain.}
\email{nero.budur@kuleuven.be}

\author{Quan Shi}
\address{Department of Mathematical Sciences, Tsinghua University, Beijing, 100084, P. R. China.}
\email{shiq24@mails.tsinghua.edu.cn / thusq20@gmail.com}

\author{Huaiqing Zuo}
\address{Department of Mathematical Sciences, Tsinghua University, Beijing, 100084, P. R. China.}
\email{hqzuo@mail.tsinghua.edu.cn}

\begin{document}

\title{Motivic principal value integrals for hyperplane arrangements}

	\begin{abstract} 
	A conjecture of Denef-Jacobs-Veys relates motivic principal value integrals of multivalued rational top-forms with cohomology support loci of rank one local systems. We give a stronger positive answer to this conjecture for hyperplane arrangements. 
	\end{abstract}

\maketitle

\setcounter{tocdepth}{1}
\tableofcontents

\setlength{\parindent}{0pt}
\setlength{\parskip}{3pt}

	\section{Introduction}
	\subsection{}
	Let $X$ be a smooth irreducible complex projective variety of dimension $n$. Consider a multivalued rational $n$-form $\omega^{1/q}$ on $X$, that is, $\omega \in (\Omega^n_{K(X)/\bC})^{\otimes q}$ with $q\in\bZ_{ \ge 1}$, where $\Omega^n_{K(X)/\bC}$ is the vector space of rational $n$-forms on $X$. Let $\textstyle\divi(\omega^{1/q})=\sum_{i\in S}(a_i-1)E_i$ be the divisor defined by $\omega^{1/q}$, with $a_i\in\frac{1}{q}\bZ$, and $E_i$ with $i\in S$ the irreducible components of the support $|\divi(\omega^{1/q})|$, so in particular $a_i\neq 1$. Then $\divi(\omega^{1/q})$ is $\bQ$-linearly equivalent to the canonical divisor $K_X$. We will assume that it is a normal crossings divisor, and that $\omega^{1/q}$ has no logarithmic poles, that is, $a_i\neq 0$ for all $i\in S$. 
		
	Two objects can now be constructed. The first  is a rank one local system $\cL(\omega^{1/q})$ on the complement $X\setminus |\divi(\omega^{1/q})|$  in the analytic topology. The sections  over a non-empty connected open subset are the local analytic branches of $\omega^{1/q}$ up to multiplication by complex numbers. See Remark \ref{rmkPictau} for a generalization of this construction.
 		
	Cohomology support loci, and more generally, cohomology jump loci of rank one local systems on smooth complex algebraic varieties admit a strong geometric and arithmetic structure, see \cite{BWabs}. On the other hand, the determination of cohomology support loci is a difficult task even in the case of complements of hyperplane arrangements, for which a folklore conjecture claims that these loci are combinatorially determined.

A second object  attached to $\omega^{1/q}$ is the {\it motivic principal value integral} defined by Veys  \cite{Motivic_Principal_Value_Integrals_Veys}, 
\be\label{eqPVI}\textstyle 
		\PV \int_X \omega^{1/q} := \mathbb L^{-n}\sum_{I\subset S} [E_I^\circ] \prod_{i\in I} \frac{\mathbb L-1}{\mathbb L^{a_i}-1}.
	\ee
Here $E_I^\circ = (\cap_{i\in I} E_i)\setminus (\cup_{j\in S\setminus I} E_j)$, and  $[\_]$ denotes the class of a variety in the Grothendieck ring $K_0(Var_\bC)$ of complex algebraic varieties, $\mathbb L$ is the class of $\mathbb \bA^1$, and (\ref{eqPVI}) is defined in a localization of an extension of $K_0(Var_\bC)$ to fractional powers of $\bL$, see \ref{subsmpvi}. 
	
	The real and $p$-adic versions of the principal value integrals were introduced by Langlands  \cite{Langlands1, Langlands2} to study orbital integrals. They  appear as coefficients of asymptotic expansions of oscillating integrals and fiber integrals, as residues of poles of distributions given by complex powers $\vert f\vert^{\lambda}$ for a polynomial $f$,  of Igusa $p$-adic zeta functions, see \cite{Distribution_3, Distribution_1, Distribution_2, On_the_Vanishing_of_Principal_Value_Integrals_Denef_Jacobs}. The motivic principal value integrals appear as residues of Denef-Loeser motivic zeta functions, see \cite{Vanishing_of_Principal_Value_Integrals_on_Surfaces}. Hence the non-vanishing of the motivic principal value integrals can help determine some poles of motivic zeta functions. The poles of the motivic zeta functions are  related by the monodromy conjecture of Igusa-Denef-Loeser with  roots of  $b$-functions and Milnor monodromy eigenvalues.

The following  was posed by Denef-Jacobs \cite{On_the_Vanishing_of_Principal_Value_Integrals_Denef_Jacobs} in a $p$-adic setup, and in this form by Veys \cite{Vanishing_of_Principal_Value_Integrals_on_Surfaces}:

\begin{conj}\label{DJV_Conjecture} Let $X$ be a smooth projective complex variety and $\omega^{1/q}$  a multivalued rational $n$-form on $X$ with no logarithmic poles and whose divisor has normal crossings support (and, maybe, all $a_i\not\in\bZ$). 
If $H^i(X\setminus|\divi(\omega^{1/q})|,\mathcal L(\omega^{1/q})) = 0$ for all $i\in\bZ$, then $\PV \int_X \omega^{1/q} = 0$.
\end{conj}

	This would relate motivic principal value integrals with cohomology support loci of rank one local systems, one of the reasons why this conjecture is interesting.

While the monodromy conjecture is known by now for a handful of classes of examples, the only non-trivial case known so far of Conjecture \ref{DJV_Conjecture} is for the case when  $X$ is a rational surface and $|\divi(\omega^{1/q})|$ is connected,  by \cite[0.4]{Vanishing_of_Principal_Value_Integrals_on_Surfaces}. In fact, he proved that less is required in this case for the vanishing of $\PV \int_X \omega^{1/q}$, namely,  that there exists a connected normal crossings divisor $B$ containing $|\divi(\omega^{1/q})|$ such that the topological Euler characteristic $\chi(X\setminus B)\le 0$.
	
	In this paper we prove Conjecture \ref{DJV_Conjecture} for the canonical log resolution of a hyperplane arrangement. Like for the case of surfaces, we  prove a stronger statement.  An {\it affine} (respectively {\it projective}) {\it hyperplane arrangement} is the union of a finite non-empty collection of hyperplanes in  $\bC^n$ (respectively, in $\bP^n$). An {\it edge} of a hyperplane arrangement, in $\bC^n$ or $\bP^n$, is any intersection of a non-empty subset of hyperplanes in the arrangement. An affine hyperplane arrangement in $\bC^n$ is {\it decomposable} if there exists a linear change of coordinates on $\bC^n$ for which a polynomial defining the arrangement can be written as the product of two non-constant polynomials in disjoint sets of variables. Otherwise the arrangement is called {\it indecomposable}. 
An edge $W$ of an affine hyperplane arrangement $\bC^n$ is called {\it dense} if the
hyperplane arrangement induced in $\bC^n/W$ is indecomposable. An edge of a projective arrangement is {\it dense} if  its cone is a dense edge of the affine cone of the arrangement.
An affine hyperplane arrangement  is {\it essential} if there is an edge of dimension zero. An affine hyperplane arrangement is {\it central} if all hyperplanes contain  $\bm 0\in\bC^n$; in this case we can form the associated projective arrangement in $\bP^{n-1}$. Conversely, the cone over a projective hyperplane arrangement in $\bP^{n-1}$ is a central affine hyperplane arrangement in $\bC^{n}$.
	
	Let $\cA$ be a hyperplane arrangement in $\bP^n$. If $X$ is the blowup of (the strict transforms of) all edges of $\cA$ inductively by increasing dimension, then the inverse image of $\cA$ is a simple normal crossings divisor in $X$ by \cite{ESV}. We call $X$ the {\it canonical log resolution} of $(\bP^n,\cA)$.

\begin{thrm}\label{thm1} Let $X$ be the canonical log resolution of a hyperplane arrangement $\cA\subset\bP^n$, and let $\omega^{1/q}$ be a multivalued rational $n$-form on $X$ without logarithmic poles and such that the support  of $\divi(\omega^{1/q})$ contains the strict transforms of all dense edges. Then Conjecture \ref{DJV_Conjecture} holds (no need to assume that all $a_i\not\in \bZ$).
\end{thrm}

This will follow from the vanishing of $\PV\int_X\omega^{1/q}$ in the case that the affine cone in $\bC^{n+1}$ over $\cA$  is a decomposable or not essential hyperplane arrangement, proven in \ref{subNndd}. The proof makes use of a particular structural result for the complement in this case, as in the case of \cite{Vanishing_of_Principal_Value_Integrals_on_Surfaces}. The difficulty with Conjecture \ref{DJV_Conjecture} in general is the current lack of structural results or classification for complements $X\setminus|\divi(\omega^{1/q})|$, and more generally for quasi-projective varieties $U$, that have  trivial topological Euler characteristic, since the cohomology of a rank-one local system on $U$ vanishes if and only if $\chi(U)=0$, see \cite[Proposition 2.5.4]{Di}.
	
Since the non-vanishing and, more generally, the determination of motivic principal value integrals is also an interesting problem, we give a few  results in this direction. 
Write $\cA\subset\bP^n$ as the projectivization of a central hyperplane arrangement $A=\cup_{i=1}^dV_i\subset\bC^{n+1}$ where $V_i$ are mutually distinct hyperplanes. Then the set of edges of $\cA$ is in bijection with the set $\cS$ of edges different than the origin of $A$. With $X$ and $\omega^{1/q}$ as in Theorem \ref{thm1}, the set $\cS$ is  in bijection with the set of irreducible components of the inverse image of $\cA$ in $X$, and we denote by $E_W$ the component corresponding to $W\in\cS$. We write $\divi(\omega^{1/q})=\sum_{W\in\cS}(a_W-1)E_W$. When $W=V_i$ for some $i\in\{1,\ldots,d\}$, we set $a_i=a_W$.
For $W\in \mathcal S$, set $b_W := \mathrm{codim}\, W+\sum_{W \subset V_i} (a_i-1)$.

\begin{prop}\label{prop1} Let $X$ be as in Theorem \ref{thm1}. Let $\omega^{1/q}$ be a multivalued rational $n$-form on $X$ without logarithmic poles whose divisor has support contained in the inverse image of $\cA$ in $X$. Then the constant term of $\mathbb L^n \cdot \PV\int_{X} \omega^{1/q}$ as a power series in $\bL^{1/q}$ is 
\be\textstyle
1+\sum_{r\ge 0}(-1)^r\cdot\#\{W_0\subsetneq \ldots  \subsetneq W_r\mid W_i\in \cS \text{ and } b_{W_i}<0 \text{ for all }i=0,\ldots, r\}.
\ee
\end{prop}

\begin{thrm}\label{prop2}
	Let $X$  be as in Theorem \ref{thm1}.  Suppose the cone over $\cA$ is essential and indecomposable. Then $\PV\int_{X} \omega^{1/q} \neq 0$ for a generic multivalued rational $n$-form $\omega^{1/q}$ without logarithmic poles  whose divisor has support contained in (or, equal to) the set-theoretic inverse image of $\cA$ in $X$.
\end{thrm}
Here by generic we mean the following. Multivalued rational $n$-forms on $X$ whose divisors have support contained in $\cup_{W\in \cS}E_W$ are in bijection with the multivalued rational $n$-forms on $\bP^n$ whose divisors have support contained in $\cA$, that is, the form is determined by the coefficients $a_i\in\bQ$ for $i=1,\ldots, d$. By linear equivalence with $K_{\mathbb P^n}$, the parameter space for such forms is $
\bQ^d\cap\{n+1-d+\textstyle\sum_{i=1}^da_i=0\}$. The subspace parametrizing multivalued rational $n$-forms  without logarithmic poles on $X$ can be then identified with $\textstyle(\bQ^d\cap\{n+1-d+\textstyle\sum_{i=1}^da_i=0\})\setminus \cup_{W\in\cS}\{b_W=0\},$ where $b_W$ are viewed as linear polynomials in $a_1,\ldots, a_d$. This space is non-empty, since $\bm 0\not\in\cS$.
If we also want to assume that the support of the divisor is exactly $\cup_{W\in \cS}E_W$, we must also remove the hyperplanes $\{b_W=1\}$ for $W\in\cS$.
This space is non-empty if the cone over $\cA$ is essential. By generic we mean in the non-empty complement of a possibly larger hyperplane arrangement. 

The strong monodromy conjecture for a polynomial $f$ predicts that  every pole of the motivic zeta function of $f$ is a root of the $b$-function of $f$, see \cite{Motivic_Zeta_Function}. For hyperplane arrangements it is shown in \cite{Monodromy_Conjecture_for_Hyperplane_Arrangement_Budur_Mustata_Teitler} that this conjecture is implied by the $n/d$-conjecture which states that for a polynomial $g\in\bC[x_1,\ldots,x_n]$ of degree $d$ whose reduced zero locus is a central essential indecomposable  hyperplane arrangement, $-n/d$ is always a root of the $b$-function. However, $-n/d$ is not always a pole of the motivic zeta function of such $g$, see the example of W. Veys in \cite[Appendix]{Local_Zeta_Function_And_b_Function_of_Certain_Hyperplane_Arragement_Budur_Saito_Sergey}. It is difficult to determine if it is a pole or not even though the motivic zeta function depends only on the  combinatorics of the arrangement.
Using the interpretation of the motivic principal value integrals as  residues of  motivic zeta functions, we show  that for generic non-reduced structures on the arrangement, the strong monodromy conjecture implies  the $n/d$-conjecture: 

\begin{cor}\label{propC2} Consider a central essential indecomposable arrangement in $\bC^n$ of $d$ mutually distinct hyperplanes defined by linear polynomials $f_i\in\bC[x_1,\ldots,x_n]$ with $i=1,\ldots, d$. 
There exists a non-zero product $\Theta$ of polynomials of degree one  in $d$ variables such that if $\bm m \in \mathbb \bZ_{>0}^d\setminus \{\Theta = 0\}$ then $-n/\sum_{i=1}^dm_i$ is a pole of the motivic zeta function of $g=f_1^{m_1}\cdots f_d^{m_d}$. 
 In particular, for these $g$, the strong monodromy conjecture implies the $n/d$-conjecture. 
 \end{cor}

While writing this paper, we realized that the $n/d$-conjecture was still open for generic multiplicities of a fixed arrangement. This is now addressed in \cite{SZ24b} by the second and third authors, where the $n/d$-conjecture is  shown for $\bm m$ outside the zero and polar loci of a meromorphic function. However, those loci are difficult to compute and it is still open whether they form a union of hyperplanes as in Corollary \ref{propC2}.

We also compute  the motivic principal value integrals for generic hyperplane arrangements. In this case $\cA$ is already a simple normal crossings divisor.

\begin{thrm}\label{prop3}
 Let $\cA$ be a generic hyperplane arrangement of degree $d$ in $\bP^n$. Let $\omega^{1/q}$ be a multivalued rational $n$-form on $\bP^n$ without logarithmic poles, whose divisor has support  $\cA$. Then:
 
 \begin{enumerate}[(a)]
 \item If $S_j$ denotes the degree $j$ elementary symmetric polynomial in $\mathbb L^{a_1},\ldots ,\mathbb L^{a_d}$, with $S_0=1$, then
		\be\label{eqProp3}\textstyle
			\PV \int_{\bP^n} \omega^{1/q} = \mathbb L^{-n}\prod_{i = 1}^d(\mathbb L^{a_i}-1)^{-1}\sum_{r = 0}^{d} (-1)^{n+r} S_{d-r}  \sum_{i = 0}^n {d-1-r \choose i} {r-1\choose n-i} \mathbb L^{n-i}. 
		\ee

\item 
$
\PV \textstyle\int_{\bP^n} \omega^{1/q}
=0$ if $1\le d\le n+1$ (cf. \cite[Prop. 6.1]{VeyR}), and  $\PV \textstyle\int_{\bP^n} \omega^{1/q}\neq 0$ if $d\ge n+2$. 

\item The same (non)vanishing holds for $\PV \textstyle\int_{X} \mu^*\omega^{1/q}$ where $\mu:X\to\bP^n$ is the blowup of all edges if has $\mu^*\omega^{1/q}$ no logarithmic poles (the support of  $\divi(\mu^*\omega^{1/q})$ does not have to be $\mu^{-1}(\cA)$).
\end{enumerate}
\end{thrm}

	\subsection{} A concern with Conjecture \ref{DJV_Conjecture} is that
 (\ref{eqPVI}) does not change if some $a_i=1$, but the cohomology is sensitive to  taking the complement in $X$ of a larger divisor than  $|\divi(\omega^{1/q})|$. So we propose a more natural conjecture, for a divisor not necessarily with normal crossings in place of a form $\omega^{1/q}$:

\begin{conj}\label{conj2}
Let  $X$ be a smooth irreducible complex projective variety. Let $E_i$ with $i\in S$ be a finite set of mutually distinct prime divisors on $X$. Let $\bm a=(a_i)_{i\in S}$
with
$a_i\in\bQ$ satisfy 
\be\label{eqDivi} \textstyle
\sum_{i\in S}(a_i-1)E_i\sim_\bQ K_X.
\ee	
Let $\cL(\bm a)$ be the rank one local system on $U=X\setminus \cup_{i\in S}E_i$ whose monodromy around each $E_i$ is given by multiplication by $e^{2\pi\sqrt{-1}a_i}$. Assume a good log resolution of $(X,\sum_i(a_i-1)E_i)$ that is an isomorphism over $U$ exists. If $H^*(U,\cL(\bm a))=0$ then $\PV(X,\bm a)=0$.
\end{conj}

For $\cL(\bm a)$, the motivic principal value integral $\PV(X,\bm a)$, and good log resolutions, see  \ref{rmkPictau} and \ref{rmkVeys}. The result of \cite{Vanishing_of_Principal_Value_Integrals_on_Surfaces} implies that Conjecture \ref{conj2} is true for $X$ a rational surface and $\cup_{i\in S}E_i$  connected, since the inverse image of $\cup_{i\in S}E_i$ in a good log resolution that does not change $U$ stays connected and has normal crossings.
We show Conjecture \ref{conj2} for hyperplane arrangements:


\begin{thrm}\label{thrmConj2}
Let $A=\cup_{i=1}^dV_i\subset\bC^{n+1}$ be a central hyperplane arrangement and $U=\bP^n\setminus \bP(A)$. Let $\bm a\in \bQ^{d}$ such that $\sum_{i=1}^da_i=d-n-1$.
 Assume that there exists a good log resolution for $(\bP^n,\sum_{i=1}^d(a_i-1)\bP(V_i))$. If $H^*(U,\cL(\bm a))=0$ then $\PV(\bP^n, \bm a)=0$.
\end{thrm}

Section \ref{secPre} is for preliminary material. In Section \ref{secGeneric} we prove Theorem \ref{prop3}. In Section \ref{secGenHA} we prove the remaining statements from this introduction.

{\it Notation.} For  $n$ in a $\bQ$-algebra, $m \in \mathbb \bZ$,  ${n \choose m}:=\frac{n(n-1)\ldots (n-m+1)}{m!}$ if $m\ge 0$, and $0$ if $m < 0$. 	
	
{\it Acknowledgement.} We thank W. Veys for discussions and the anonymous referee for various corrections. N.B. was  supported by a Methusalem grant and G0B3123N from FWO.	H.Z. and Q.S. was supported by BJNSF Grant 1252009. Q.S. was supported by NSFC Grant 125B2004. H.Z. was supported by NSFC Grant 12271280.

	\section{Preliminaries}\label{secPre}

	\subsection{Multivalued rational $n$-forms and local systems}\label{subsection_Multivalued_n_form}
	
	Let $X$ be a smooth irreducible complex projective variety of dimension $n$. Let $\omega^{1/q}$ be a multivalued rational $n$-form on $X$,  with divisor 
$
\sum_{i\in S}(a_i-1)E_i\sim_\bQ K_X,$	$ a_i\in\bQ,
$	
where $E_i$ are mutually distinct prime divisors.  Let  $E=\cup_{i\in S} E_i$ and $U=X\setminus E$.

We saw in the introduction how one can associate a rank one local system $\cL(\omega^{1/q})$ on $U$ under the assumptions that $E$ is a normal crossings divisor and $a_i\neq 1$ for all $i\in S$, that is, $E$ is the support of the divisor of $\omega^{1/q}$, and also a motivic principal valued integral $\PV\int_X\omega^{1/q}$ under a further assumption that $a_i\neq 0$ for all $i\in S$. This can be generalized as follows.

\begin{rmk}\label{rmkPictau} (a)  Suppose we are given only the data  (\ref{eqDivi}), that is, a finite set of mutually distinct prime divisors $E_i$ on $X$ and rational numbers $a_i$, $i\in S$, satisfying the condition (\ref{eqDivi}). To this data, we associate a rank one local system $\cL(\bm a)$ on $U$ as follows. The set
$$\textstyle
\Pic^\tau(X,E)=\{ (L,\bm\alpha)\in \Pic(X)\times [0,1)^S\mid c_1(L)=\sum_{i\in S}\alpha_i[E_i]\in H^2(X,\bR)\}.
$$
 is endowed with a natural group structure; and, the map $\Pic^\tau(X,E)\to \Hom(H_1(U,\bZ),S^1)$ sending $(L,\bm\alpha)$ to the unitary rank one local system on $U$ with monodromy around a generic point of $E_i$ (always meant along a loop bounding a small normal disc to $E_i$, centered at the point, in counter-clockwise direction using the natural orientation) given by multiplication by $e^{2\pi \sqrt{-1}\al_i}$ for all $i\in S$, is a group isomorphism, by \cite[Theorem 1.2]{B-uls}. 
 Our local system $\cL(\bm a)$ corresponds to the pair $(\cO_X(K_X-\sum_{i\in S}\lfloor a_i-1\rfloor E_i), (\{a_i-1\})_{i\in S})$, where $\lfloor \_\rfloor$, $\{\_  \}$ are the round-down and, respectively, fractional part. 
Hence $\cL(\bm a)$ is uniquely determined by declaring that the monodromy around a generic point of $E_i$ is the multiplication by $e^{2\pi \sqrt{-1}a_i}$.

(b)
Furthermore, the isomorphism $\Pic^\tau(X,E)\xra{\sim} \Hom(H_1(U,\bZ),S^1)$ is compatible with different log resolutions of $(X,E)$. Let $\mu:X'\to X$ be a log resolution of $(X,E)$ that is an isomorphism above $U$. Denote by $E_j'$ with $j\in S'$ the irreducible components of $\mu^{-1}(E)$. Then the map
$$\textstyle
\Pic^\tau(X,E) \to \Pic^\tau(X',E'),\quad 
(L,\bm\al)\mapsto (\mu^*L\otimes_{\cO_{X'}}\cO_{X'}(-\sum_{j\in S'}\lfloor \beta\rfloor E'_j),\bm\beta)
$$
is a group isomorphism, where $\mu^*(\sum_{i\in S}\al_iE_i)=\sum_{j\in S'}\beta_jE'_j$, by \cite{B-uls}.

(c) If $E$ is a normal crossings divisor and $\cL(\omega^{1/q})$ is as in the introduction with $a_i\neq 1$, namely, $\cL(\omega^{1/q})$ is defined by declaring that its sections  over a non-empty connected open subset are the local analytic branches of $\omega^{1/q}$ up to multiplication by complex numbers, then $\cL(\bm a)\cong \cL(\omega^{1/q})$, as it can be seen from the local monodromies. 
\end{rmk}

\subsection{Motivic principal value integrals.}\label{subsmpvi}	The following generalization is due to Veys \cite{Motivic_Principal_Value_Integrals_Veys}:

\begin{rmk}\label{rmkVeys}
With the $X$, $(E_i)_{i\in S}$, and $\bm a=(a_i)_{i\in S}$ as in \ref{subsection_Multivalued_n_form}, suppose \ that there exists a {\it good log resolution} $\mu:X'\to X$ of $(X,\sum_{i\in S}(a_i-1)E_i)$, that is, a log resolution of $(X,E)$ such that all the coefficients of the divisor $K_{X'/X}+\mu^*(\sum_{i\in S}(a_i-1)E_i)$ are $\neq -1$. Here, and elsewhere in this paper, we do not require that a log resolution $\mu$ has to be an isomorphism over $U$, unless mentioned explicitly. Then the expression on the right-hand side of (\ref{eqPVI})  for $X'$ is well-defined in the ring $\mathcal R_q$, see below,  and it is independent of the choice of a good log resolution by \cite[2.6]{Motivic_Principal_Value_Integrals_Veys}. We shall denote this by $\PV(X,\bm a)$.
If  $\sum_{i\in S}(a_i-1)E_i=\divi(\omega^{1/q})$ for some multivalued rational $n$-form $\omega^{1/q}$ on $X$, then a log resolution is good iff $\mu^*\omega^{1/q}$ has no logarithmic pole, and so the motivic principal value integral $\PV\int_{X'}\mu^*\omega^{1/q}$ is well-defined, and it is independent of such resolutions. If we  denote this by $\PV\int_{X}\omega^{1/q}$, then $\PV\int_{X}\omega^{1/q}=\PV(X,\bm a)$. 
\end{rmk}

 We now assume for the rest of the section that $E$ is a simple normal crossings divisor, equal to the support of the divisor of a  multivalued rational $n$-form $\omega^{1/q}$ without logarithmic poles. 
 
 \begin{rmk}
 The motivic principal value integral
 $\PV \int_X \omega^{1/q}$ is defined in the subring $\mathcal R_q$ of the localization of $K_0(Var_\bC)[T]/(\bL T^q-1)$ with respect to the elements $T=:\bL^{-1/q}$ and $\bL^{i/q}-1$ for $i\in\bZ\setminus\{0\}$, generated by $K_0(Var_\bC)$, $\bL^{-1}$, $(\bL-1)/(\bL^{i/q}-1)$ for $i\in\bZ\setminus\{0\}$.
 \end{rmk}


				\begin{rmk}\label{rmkClo}	
Using the inclusion-exclusion principle one has
$$
\PV \int_X \omega^{1/q} \textstyle= \mathbb L^{-n}\sum_{J\subset S} [E_J] \prod_{i\in J} \left(\frac{\mathbb L-1}{\mathbb L^{a_i}-1}-1\right).
$$
Indeed, the left-hand side is
$$ 		\textstyle \mathbb L^{-n}\sum_{I\subset S} [E_I^\circ] \prod_{i\in I} \frac{\mathbb L-1}{\mathbb L^{a_i}-1} = \mathbb L^{-n}\sum_{I\subset S} \left(\sum_{J\supset I} (-1)^{\vert J\vert -\vert I\vert} [E_J] \right) \prod_{i\in I} \frac{\mathbb L-1}{\mathbb L^{a_i}-1}$$
$$ \textstyle = \mathbb L^{-n} \sum_{J \subset S}[E_J] \left( \sum_{I \subset J} (-1)^{\vert J\vert - \vert I\vert} \prod_{i\in I}  \frac{\mathbb L-1}{\mathbb L^{a_i}-1}  \right)$$
which equals the right-hand side of the equality claimed above.
 	\end{rmk}

	\begin{rmk}[MPVI as residues]\label{rmkMPIVres}
	
	Motivic principal value integrals arise sometimes as residues of motivic zeta functions. 
	Let $f \in \mathbb C[x_0,\ldots ,x_n]\setminus \mathbb C$. Let $\mu : X \to \mathbb C^{n+1}$ be a log resolution of $f^{-1}(0)$. Let $D_i$ with $i\in S'$ be the irreducible components of $(f\circ\mu)^{-1}(0)$. Write $K_{X/\mathbb C^{n+1}} = \sum_{i\in S'}(\nu_i-1) D_i$ and $\mu^*(\mathrm{div}\, f) = \sum_{i\in S'} N_i D_i$ for  $\nu_i, N_i\in\bZ$. One should be careful that $S'$ differs from $S$ in \eqref{eqPVI}, where the ambient space is projective.
	The motivic zeta function of $f$ admits the  formula
$$ 		\textstyle	Z_f^{\mathrm{mot}}(s) = \mathbb L^{-n}\sum_{I\subset S'} [D_I^\circ] \prod_{i\in I} \frac{\mathbb L-1}{\mathbb L^{N_is+\nu_i}-1}
$$ 		where $D_I^\circ = (\bigcap_{i\in I} D_i)\setminus (\bigcup_{j\not\in I} D_j)$ in an appropriate localization of the ring $K_0(Var_\bC)[\bL^{-s}]$ by \cite{Motivic_Zeta_Function}. For $I = \emptyset$, we set $D_{\emptyset}^\circ = X\setminus (\bigcup_{j\in S'} D_j)$. If after cancelations, $\bL^{N_is+\nu_i}-1$ still survives in the denominator, one says that $-\nu_i/N_i$ is a {\it pole} of $Z_f^{\mathrm{mot}}(s)$.
 	We say that $D_i$ is  {\it numerically generic} if  $\nu_i/N_i \neq \nu_j/N_j$ for all $j\neq i$. Under this assumption, $-\nu_i/N_i$ is a pole of $Z_f^{\mathrm{mot}}(s)$ if and only if its ``residue" does not vanish, that is,
$$ 	\textstyle	R_{i} := \sum_{i\in I\subset S} [D_I^\circ] \prod_{ j \in I\setminus\{i\}} \frac{\mathbb L-1}{\mathbb L^{\alpha_j}-1} \neq 0, 
$$ 	where $\alpha_j = \nu_j-\frac{\nu_i}{N_i} N_j$.   Let $S_i = \{j\in S\setminus \{i\} \mid D_i\cap D_j = \emptyset\}$. Then $R_{i}=\PV\int_{D_i}\eta_i$ for a multivalued rational $n$-form $\eta_i$ on $D_i$ with divisor  $\sum_{j\in S_i} (\alpha_j-1) (D_j\cap D_i)$. Namely,  $\eta_i$ is the Poincar\'e residue of $(f\circ\mu)^{-\nu_i/N_i}\mu^*(dx_0\wedge \ldots \wedge dx_n)$, see \cite{Vanishing_of_Principal_Value_Integrals_on_Surfaces}.
\end{rmk}

	\subsection{Hyperplane arrangements}\label{subsection_canonical_log_resolution}
	
	Let $A = \cup_{i=1}^d V_i \subset \mathbb C^{n+1}=:V$ be a central  hyperplane arrangement with degree $d$. Recall that $A\subset V$ is indecomposable, as defined in the introduction, if and only if $\chi(\mathbb P(V)\setminus \mathbb P(A)) \neq 0$, by \cite{Resolution_of_Hyperplane_Arrangements_Schchtman_Terao_Varchenko}. 
	Let $\mathcal S$ be the set of edges of $A$ different than the origin $\bm 0$ in $V$, as in the introduction, so that $S$ is in bijection with the set of edges of the projectivized arrangement $\bP(A)$. Denote by $$\mu : \mathbb P(V)^{\mathcal S} \to \mathbb P(V)$$ the canonical log resolution of $(\mathbb P(V),\mathbb P(A))$ obtained by blowing up, inductively in increasing order of dimensions, the (strict transforms of the) edges of $\bP(A)$. The irreducible components of $\mu^{-1}(\bP(A))$ are in bijection with the edges in $\cS$. We let $E_W$ be the component corresponding to $W\in \cS$. For  $I\subset \cS$ we set $E_I=\cap_{W\in I}E_W$ and $E_I^\circ=E_I\setminus\cup_{W\in\cS\setminus I}E_W$.

For an edge $W \in \mathcal S$,  let 
	$$ \mathcal S_W := \{W'\in \mathcal S \mid W\subsetneq W'\},\quad
		 \mathcal S^W := \{W' \in \mathcal S \mid W'\subsetneq W\}$$
		 the latter equal to the set of edges different than the origin of the  hyperplane arrangement induced by $A$ in  $W$. Set $\mathcal S^{V} = \mathcal S_{\bm 0} = \mathcal S$. For $ W_1, W_2 \in  \mathcal S\cup \{\bm 0,V\}$ with $ W_1 \subsetneq  W_2$,  let
	\begin{displaymath}
		\mathcal S_{ W_1}^{ W_2} := \{W'/ W_1 \mid W' \in \mathcal S_{W_1}\cap \mathcal S^{ W_2}\},
	\end{displaymath}
	 the set of edges different than the origin of the induced hyperplane arrangement  in $ W_2/ W_1$.

	\begin{proposition}\label{propBS} \cite[2.7]{Jumping_Coefficient_and_Spectrum_of_a_Hyperplane_Arrangement_Budur_Saito} \label{Budur-Saito_Description_of_E_I}
		Suppose $I = \{W_1,\ldots ,W_r\} \subset \mathcal S,r = \vert I\vert$, then $E_I\neq \emptyset$ if and only if $W_1\subsetneq \ldots  \subsetneq W_r$ after some permutation. If $W_1\subsetneq \ldots  \subsetneq W_r$,  let $W_0 =  \bm 0 $ and $W_{r+1} = V$, then
$			E_I = \prod_{i = 1}^{r+1} \mathbb P(W_i/W_{i-1})^{\mathcal S_{W_{i-1}}^{W_{i}}}.
$	\end{proposition}
	
 Define		$ U(\mathcal S) := \mathbb P(V)\setminus \bP(A) \simeq \mathbb P(V)^{\mathcal S} \setminus\cup_{W\in \mathcal S} E_W.$

	\begin{corollary}\label{E_I_circ forumla}
		With notation of Proposition \ref{Budur-Saito_Description_of_E_I}, if $W_1\subsetneq \ldots  \subsetneq W_r$ then
$	[E_I^\circ] = \prod_{i = 1}^{r+1} [U({\mathcal S_{W_{i-1}}^{W_{i}}})] 
$ in $K_0(Var_\bC)$.	\end{corollary}
\begin{proof}	
	By Proposition \ref{propBS}, subsets $J \subset \mathcal S$ with $I \subset J$ and $E_J \neq \emptyset$ correspond to chains refining the chain $W_1 \subsetneq \ldots \subsetneq W_r$, that is, to chains of the form $\mathcal W_1 \subsetneq W_1 \subsetneq \mathcal W_2 \subsetneq W_2 \subsetneq \ldots \subsetneq W_r \subsetneq \mathcal W_{r+1}$, where each $\mathcal W_i$  is an increasing chain of edges in $\mathcal S_{W_{i-1}} \cap \mathcal S^{W_{i}}$, possibly empty. Applying the inclusion–exclusion principle, we obtain:
		\begin{align*}
			[E_I^\circ] & \textstyle = \sum_{J\supset I} (-1)^{\vert J\vert -\vert I\vert}[E_J] = \sum_{\mathcal W_1,\ldots,\mathcal W_{r+1}} \prod_{i = 1}^{r+1}  \left((-1)^{\vert \mathcal W_i\vert} \prod_{j=1}^{\vert \mathcal W_i\vert+1} [\mathbb P(W_{i,j}/W_{i,j-1})^{\mathcal S_{W_{i,j-1}}^{W_{i,j}}}] \right) \\
			&\textstyle = \prod_{i = 1}^{r+1} \sum_{\mathcal W_i} \left((-1)^{\vert \mathcal W_i\vert} \prod_{j=1}^{\vert \mathcal W_i\vert+1} [\mathbb P(W_{i,j}/W_{i,j-1})^{\mathcal S_{W_{i,j-1}}^{W_{i,j}}}] \right),
		\end{align*}
		where $\mathcal W_i = \{W_{i,1} \subsetneq \ldots \subsetneq W_{i,|\mathcal W_i|}\}$ ranges over all increasing chains in $\mathcal S_{W_{i-1}} \cap \mathcal S^{W_i}$, with $W_{i,0} = W_{i-1}$ and $W_{i,|\mathcal W_i| + 1} = W_i$. Denote the expression in brackets in the last term of the displayed equality above  by $\mathcal Q_{i,\mathcal W_i}$.
	
Now fix $i \in \{1, \ldots, r+1\}$. Let $V' = W_i / W_{i-1}$ and $\mathcal S' = \mathcal S_{W_{i-1}}^{W_i}$. For any $I' \subset \mathcal S'$, let $E_{I'}'$ denote the corresponding divisor in $\mathbb{P}(V')^{\mathcal S'}$. Then:
		\begin{displaymath}\textstyle
			\sum_{\mathcal W_i} \mathcal Q_{i,\mathcal W_i} = \sum_{I' \subset S'} (-1)^{\vert I'\vert} [E_{I'}'] = [\mathbb P(V')^{\mathcal S'}] - [\cup_{i'\in \mathcal S'} E_{\{i'\}}'] = [U(\mathcal S_{W_{i-1}}^{W_i})],
		\end{displaymath}
		again by the inclusion–exclusion principle. This completes the proof.
\end{proof}

An immediate consequence is the following. Let $\mathcal S_q$ be the subring of $\mathcal R_q$ generated by $\bL, \bL^{-1}, (\bL-1)/(\bL^{i/q}-1)$ for $i\in\bZ\setminus\{0\}$. Then we have an isomorphism between $\mathcal S_q$ and the subring $\bZ[t,t^{-1}][(t-1)/(t^{i/q}-1)\mid i\in\bZ\setminus\{0\}]\subset \bQ(t^{1/q})$ of the function field in one variable $t^{1/q}$ over $\bQ$, given by $\bL\mapsto t$, as one can see using the virtual Poincar\'e, or Hodge-Deligne, specialization. Note also that  $\bQ(\bL^{1/q})\cong \bQ(t^{1/q})$, so  we have an injection $\mathcal S_q\subset\bQ(\bL^{1/q}).$ Also note that
\begin{equation}
	\frac{1}{\mathbb L^a-1} = 
	\begin{cases}
		-(1+\mathbb L^a+\mathbb L^{2a}+\ldots), & a>0,\\
		\mathbb L^{-a}(1+\mathbb L^{-a}+\mathbb L^{-2a}+\ldots), & a<0,
	\end{cases}
	\label{eq*}
\end{equation}
is well-defined in $\bZ\llbracket \bL^{1/q}\rrbracket$ if $0\neq a\in q^{-1}\bZ$. Since $U({\mathcal S_{W_{i-1}}^{W_{i}}})$ is the complement in an affine space of a hyperplane arrangement, by inclusion-exclusion we have that is class $[U({\mathcal S_{W_{i-1}}^{W_{i}}})]$ is a polynomial in $\bL$, hence it lies in $\mathcal S_q$. Then Corollary \ref{E_I_circ forumla} implies:

\begin{cor}\label{corDefRing} Let $\omega^{1/q}$ be a multivalued rational $n$-form on $\mathbb P(V)^{\mathcal S}$ without logarithmic poles whose divisor has support in $\cup_{W\in\cS}E_W$. Then the motivic principal value integral $\PV\int_{\mathbb P(V)^{\mathcal S}}\omega^{1/q}$ lies in the subring $\mathcal S_q$ of $\bQ(\bL^{1/q})$. Moreover, $\bL^n\cdot\PV\int_{\mathbb P(V)^{\mathcal S}}\omega^{1/q}$ lies in $\bZ\llbracket \bL^{1/q}\rrbracket$.
\end{cor}

	\section{Generic hyperplane arrangements}\label{secGeneric}
	
	In this section we prove Theorem \ref{prop3}. 
	Let $\cA$ be a generic hyperplane arrangement of degree $d$ in $\bP^n$. Let $\omega^{1/q}$ be a multivalued rational $n$-form on $\bP^n$ without logarithmic poles and whose divisor has support exactly $\cA$. Then $\cA$ is a simple normal crossings divisor. Let $E_i$ be the irreducible components of $\cA$ with $i\in S=\{1,\ldots, d\}$. Write $\mathrm{div}(\omega^{1/q}) = \sum_{i = 1}^d (a_i-1) E_i$, so that $a_i \neq 0,1$ are in $\frac{1}{q}\bZ$.  Since $\deg K_{\mathbb P^n} = -n-1$, we have $\sum_{i=1}^d a_i = d-n-1$.

\begin{proof}[Proof of Theorem \ref{prop3} (a)] 	Using Remark \ref{rmkClo}, in the ring $\mathcal S_q$ we have 
		$$\textstyle
			 \mathbb L^n \cdot \PV \int_{\bP^n} \omega^{1/q}  \stackrel{\text{*}}{=} \sum_{I\subset S} [\mathbb P^{n-\vert I \vert}] \prod_{i\in I}\left(\frac{\mathbb L-1}{\mathbb L^{a_i}-1}-1\right) \stackrel{\text{**}}{=} \sum_{l = 0}^{n} [\mathbb P^{n-l}]\sum_{i=0}^{l} (-1)^{l-i} \frac{{d\choose l}{l\choose i}}{{d\choose i}} T_i$$
			$$\textstyle = \sum_{l = 0}^{n} [\mathbb P^{n-l}]\sum_{i=0}^{l} (-1)^{l-i} {d-i\choose l-i} T_i.
		$$ 
		where $T_i$ is the elementary symmetric polynomial of degree $i$ in $\frac{\mathbb L- 1}{\mathbb L^{a_1}-1},\ldots ,\frac{\mathbb L-1}{\mathbb L^{a_d}-1}$. The equality $(*)$ holds because the divisor $\cup_{i} E_i$ has simple normal crossings, and $[E_I]$ equals $[\mathbb{P}^{n-|I|}]$ if $|I| \leq n$, and $0$ otherwise. For the equality $(**)$ note that each product $\prod_{i\in I}\left(\frac{\mathbb{L}-1}{\mathbb{L}^{a_i}-1}-1\right)$ contributes to ${|I| \choose i}$ terms in $T_i$ with a sign $(-1)^{|I| - i}$. Moreover, for a fixed $|I|$, there are ${d \choose |I|}$ such products, and $T_i$ itself consists of ${d \choose i}$ terms, from which $(**)$ follows.

		Let $S_r$ be the elementary symmetric polynomial of degree $r$ in $\mathbb L^{a_1},\ldots ,\mathbb L^{a_d}$. Let $\tau = \prod_{i=1}^d(\mathbb L^{a_i}-1)^{-1}$. 
	Then
		$$\textstyle
			T_i   = (\mathbb L-1)^i\tau\sum_{r = 0}^{d-i} (-1)^{d-i-r} \frac{{d\choose d-i} {d-i\choose r}}{{d \choose r}} S_r =  (\mathbb L-1)^i\tau\sum_{r = 0}^{d-i} (-1)^{d-i-r} {d-r\choose d-i-r} S_r.
		$$
		Here, while $T_i$ lies in the ring $\mathcal S_q$, the individual terms in this expression lie in the ring extension $\bQ(\bL^{1/q})$ of $\cS_q$, isomorphic to the field of rational functions in one variable over $\bQ$.
		
		Hence  $\mathbb L^n \cdot \PV \int_{\bP^n} \omega^{1/q}$ is equal to
		\begin{align*}
			 &  \textstyle \tau\sum_{l = 0}^{n} [\mathbb P^{n-l}]\sum_{i=0}^{l} (-1)^{l} {d-i\choose l-i}(\mathbb L-1)^i\sum_{r = 0}^{d-i} (-1)^{d-r} {d-r\choose d-i-r} S_r = \\
			& \textstyle= \tau\sum_{l = 0}^{n} [\mathbb P^{n-l}]\sum_{i=0}^{l} (-1)^{l} {d-i\choose l-i}(\mathbb L-1)^i\sum_{r = i}^{d} (-1)^{r} {r\choose i} S_{d-r} \\
			& \textstyle= \tau\sum_{r = 0}^{d} (-1)^r S_{d-r} \cdot \sum_{l = 0}^n\sum_{i = 0}^{\min\{r ,l\}} [\mathbb P^{n-l}] (-1)^l  {d-i\choose l-i}{r\choose i} (\mathbb L-1)^i \\
			&\textstyle = \tau\sum_{r = 0}^{d} (-1)^r S_{d-r} \cdot \sum_{l = 0}^n\sum_{i = 0}^{l} [\mathbb P^{n-l}] (-1)^l  {d-i\choose l-i}{r\choose i} (\mathbb L-1)^i. 
		\end{align*}
		So we have to compute the term
		\begin{align*}
			& \textstyle \sum_{l = 0}^n\sum_{i = 0}^{l} [\mathbb P^{n-l}] (-1)^l  {d-i\choose l-i}{r\choose i} (\mathbb L-1)^i = \sum_{i = 0}^n (\mathbb L-1)^i {r\choose i}\sum_{l = i}^n [\mathbb P^{n-l}](-1)^l {d-i\choose l-i}\\ 
			& \textstyle= \sum_{i = 0}^n (\mathbb L-1)^i {r\choose i}\sum_{l = i}^n \sum_{j = 0}^{n-l} \mathbb L^j (-1)^l {d-i\choose l-i}  = \sum_{i = 0}^n (\mathbb L-1)^i {r\choose i} \sum_{j = 0}^{n-i} \mathbb L^j \cdot\sum_{l = i}^{n-j}  \mathbb  (-1)^l {d-i\choose l-i}\\
			& \textstyle= \sum_{i = 0}^n (\mathbb L-1)^i {r\choose i} \sum_{j = 0}^{n-i} \mathbb L^j (-1)^{n-j}{d-1-i \choose n-j-i} \\ 
			& \textstyle= \sum_{i = 0}^n{r\choose i} \sum_{j = 0}^{n-i}  \sum_{k =0}^i  (-1)^{n-j+i-k} {i\choose k} {d-1-i \choose n-j-i} \mathbb L^{j+k}  \\
			& \textstyle= \sum_{i = 0}^n {r\choose i} \sum_{m = 0}^{n}\; \sum_{k = \max\{0,m-n+i\}}^{\min\{m,i\}}  (-1)^{n+i-m} {i\choose k} {d-1-i \choose n-(m-k)-i} \mathbb L^{m}  \\ 
			&\textstyle = \sum_{m = 0}^{n} \mathbb L^{m} \sum_{i = 0}^n\;\sum_{k = \max\{0,m-n+i\}}^{\min\{m,i\}}  (-1)^{n+i-m} {r\choose i}{i\choose k} {d-1-i \choose n-(m-k)-i}. 
		\end{align*}
Replacing $k$ by $a = m-k$ this becomes		
		$$\textstyle
		 \sum_{m = 0}^{n} \mathbb L^{m} \sum_{i = 0}^n\;\sum_{a = \max\{0,m-i\}}^{\min\{m,n-i\}}  (-1)^{n+i-m} {r\choose i}{i\choose m-a} {d-1-i \choose d-1-n+a}
		$$
		$$ \textstyle= \sum_{m = 0}^{n} \mathbb L^{m}\sum_{i = 0}^n\sum_{a = 0}^{m}  (-1)^{n+i-m} {r\choose i}  {i\choose m-a} {d-1-i \choose d-1-n+a}. 
$$		

Let
$
G(n,m,d,r):= \sum_{i = 0}^n\sum_{a = 0}^{m}  (-1)^{n+i-m} {r\choose i}{i\choose m-a} {d-1-i \choose d-1-n+a}.
$
		To show (\ref{eqProp3}) it  suffices to prove that
\be\label{eqCo}		\textstyle
	G(n,m,d,r)=  (-1)^n{d-1-r \choose n-m} {r-1\choose m}.
	\ee
		 If $0\leq n,m \leq 1$, one can easily check the equality (\ref{eqCo}). To prove (\ref{eqCo}) in general we use induction. 
		 Define first  $F(n,m,d,i) := \sum_{a = 0}^{m}{i\choose m-a} {d-1-i \choose d-1-n+a}.$
	One can easily check that	 		
		$$F({n,m,d,i}) = F(n-1,m,d-1,i-1)+F(n-1,m-1,d-1,i-1),\quad\text{ for } n,m,d,i\geq 1.$$
				Then 
		$$\textstyle G(n,m,d,r) 
			 = (-1)^{n-m}{d-1\choose n-m} + \sum_{i = 0}^{n-1} (-1)^{n-1+i-m} {r \choose i+1}F(n,m,d,i+1) $$ 
			 $$\textstyle = (-1)^{n-m}{d-1\choose n-m} + \sum_{i = 0}^{n-1} (-1)^{n-1+i-m} {r \choose i+1}\bigl(F(n-1,m,d-1,i)+F(n-1,m-1,d-1,i)\bigr).
		$$
		Doing  the same computation for $G(n-1,m-1,d-1,r-1)$ and $G(n-1,m,d-1,r-1)$, we see that 
		\begin{align*}
			& G(n,m,d,r) - G(n-1,m,d-1,r-1) + G(n-1,m-1,d-1,r-1) = G(n,m,d,r-1)
		\end{align*}
		Then by the induction hypothesis, to prove (\ref{eqCo}) it suffices to show that
		\begin{align*}
			& \textstyle{d-1-r \choose n-m} {r-1\choose m} + {d-1-r \choose n-m-1} {r-2\choose m} - {d-1-r \choose n-m} {r-2\choose m-1} =  {d-r \choose n-m} {r-2\choose m}
		\end{align*}
		One  checks easily that this holds. 
	\end{proof}

\begin{proof}[Proof of Theorem \ref{prop3} (b)]  If $1\le d \leq n+1$ then
$			{d-1-r \choose i} {r-1\choose n-i} = 0
$		for all $0\leq r \leq d-1$ and $0\leq i\leq n$. From (\ref{eqProp3}) it follows that $
\PV \textstyle\int_{\bP^n} \omega^{1/q}
=0$ in this case, also shown by \cite[Prop. 6.1]{VeyR}.

	Assume now that $d \geq n+2$.	Since  $\PV_{\bP^n} \int \omega^{1/q}$ lives in the ring $\mathcal S_q\subset \bQ(\bL^{1/q})$, it is non-zero iff 
		\be\label{eqA} \textstyle\sum_{r = 0}^{d} (-1)^{n+r} S_{d-r} \cdot\textstyle \sum_{i = 0}^n {d-1-r \choose i} {r-1\choose n-i} \mathbb L^{n-i}  \neq 0 
		\ee
in $\bQ(\bL^{1/q})$. We show that there is a unique term, with a non-zero coefficient, on the left-hand side of (\ref{eqA}) containing the highest power of $\mathbb L^{1/q}$. As a consequence, the non-vanishing (\ref{eqA}) holds. 
A first observation is that
		\begin{displaymath}
			\deg_{\mathbb L} \biggl(\textstyle \sum_{i = 0}^n {d-1-r \choose i} {r-1\choose n-i} \mathbb L^{n-i} \biggr) = 
			\begin{cases}
				n, & \text{ if } r = 0 \text{ or } r \geq n+1,\\
				r-1, & \text{ if } 1\leq r \leq n.
			\end{cases}
		\end{displaymath}
			We may assume that
$			a_1 \geq a_2 \geq \ldots  \geq a_t > 0 > a_{t+1} \geq \ldots  \geq a_d
$ for some index $t$.		
		The candidate term with the largest degree in $S_{d-r}  \sum_{i = 0}^n {d-1-r \choose i} {r-1\choose n-i} \mathbb L^{n-i}$ is: 
	 $\mathbb L^{d-1}$ if $r = 0$,
	$\mathbb L^{a_1+\ldots +a_{d-r}+r-1}$ if $1\leq r\leq n$,
	$\mathbb L^{a_1+\ldots +a_{d-r}+n}$ if $r\geq n+1$.
				
		Suppose first that  $d-t \geq n+1$.
		In this case, since  $a_1+\ldots +a_t = d-n-1-(a_{t+1}+\ldots +a_d) > d-n-1$, $\mathbb L^{a_1+\ldots +a_t+n}$ is the unique term on the left-hand side of (\ref{eqA}) with the largest degree. 
		
		Suppose now that $d-t \leq n$.
	 If $a_1 < 1$, then $a_1+\ldots +a_{d-r}+r-1 < d-1$ and $a_1+\ldots +a_{d-r}+n < d-1$, so $\mathbb L^{d-1}$ is the unique term with the largest degree on the left-hand side of (\ref{eqA}).
		 If $a_1 > 1$, then
			$a_1 \geq \ldots  \geq a_p > 1 > a_{p+1} \geq \ldots $
for some index $p$.		Then $\mathbb L^{a_1+\ldots+a_p+d-p-1}$ is the unique term with the largest degree on the left-hand side of (\ref{eqA}) if $d-n\le p\le t$. If $p<d-n$ then $\bL^{d-1}$ is the unique term with the largest degree on the left-hand side of (\ref{eqA}). If $a_1=1$ we can have more than one term of largest degree, but we have assumed at the support of the divisor of the $n$-form is the whole arrangement, so this case is excluded.
		
		In conclusion, the non-vanishing (\ref{eqA}), and hence the non-vanishing of $\PV \int_{\bP^n} \omega^{1/q}$, holds.
		The non-vanishing of $\PV \int_{X} \mu^*\omega^{1/q}$ if $\mu^*\omega^{1/q}$ has no logarithmic poles follows from Remark \ref{rmkVeys}. 
\end{proof}

\section{General hyperplane arrangements}\label{secGenHA}
	
We prove now the remaining claims from the introduction.
	 Let $A = \bigcup_{i=1}^d V_i \subset V:=\mathbb C^{n+1}$ be a central  hyperplane arrangement with degree $d\ge 1$, $\cS$  the set of edges of $A$ different than the origin, and $\mu : \mathbb P(V)^{\mathcal S} \to \mathbb P(V)$  the canonical log resolution of $(\bP(V),\bP(A))$. We use the notation as in \ref{subsection_canonical_log_resolution}.

	\subsection{The constant term.}\label{a Non-Vanishing Condition for PVI of General Hyperplane Arrangements} In this subsection we prove Proposition \ref{prop1}. We start with three  lemmas.

	\begin{lemma}\label{Inclusion-Exclusion for arbitrary set}
		Let $B$ be a set. Let  $\mathcal C$ be a finite family of subsets of $B$ closed under intersection. Let $D = \cup_{C\in \mathcal C} C$. Then
$			\mathbbm{1}_{D} = \sum_{r\geq 0}\;\sum_{{C_0\subsetneq  \ldots \subsetneq C_r, C_i \in \mathcal C}}  (-1)^r \cdot \mathbbm{1}_{C_0}.
$	\end{lemma}
	\begin{proof} Here $\mathbbm{1}_{D}:B\to \{0,1\}$ is the indicator function $x\mapsto 1$ iff $x\in D$.
		For a point $x\in D$, the evaluation of the right-hand side at $x$ is
$			\sum_{r\geq 0}\;\sum_{{x\in C_0\subsetneq  \ldots \subsetneq C_r, C_i \in \mathcal C}}  (-1)^r.
$		We prove by induction on $\vert\mathcal C\vert$ that this equals $1$. It is trivial for $\vert \mathcal C\vert = 1$. We can assume that $x$ is contained in all $C \in \mathcal C$, otherwise we can exclude all sets in $\mathcal C$ that do not contain $x$ without affecting the sum. Define
$			S(\mathcal C) := \sum_{r\geq 0}\sum_{{C_0\subsetneq  \ldots \subsetneq C_r, C_i \in \mathcal C}}  (-1)^r.
$		We   show $S(\mathcal C) = 1$ as long as $\mathcal C\neq \emptyset$. Suppose $\vert \mathcal C\vert > 1$. Let $C$ be a maximal element of $\mathcal C_0$ under inclusion. Take $\mathcal C' = \mathcal C\setminus \{C_0\}$, $\mathcal C'' = \{C' \in \mathcal C \mid C'\subsetneq C_0\}$. The  increasing sequences in $\mathcal C$ are of three types: ending at $C_0$; ending at a proper subset of $C_0$; and ending at a set not contained in $C_0$.  Since $x\in \cap_{C\in \mathcal C} C$ and $\vert \mathcal C\vert > 1$, we have $\mathcal C'' \neq \emptyset$. Then 
$			S(\mathcal C) = \bigl(1-S(\mathcal C'') \bigr) + S(\mathcal C'')  + \bigl(S(\mathcal C')-S(\mathcal C'')\bigr) = 1,
$		by the induction hypothesis. 
	\end{proof}

	\begin{lemma}\label{Inclusion-Exclusion for edge sets} Let $\mathcal C$ be a finite family of linear subspaces in $\mathbb P^n$ closed under intersection.
		Let $D := \bigcup_{C\in \mathcal C} C$. In $K_0(Var_\bC)$ one has
$
[D] = \sum_{r = 0}^n \;\sum_{{C_0\subsetneq \ldots \subsetneq C_r, C_i\in \mathcal C}}  (-1)^r[C_0]. 
$
	\end{lemma}
	\begin{proof}
		 For $C \in \mathcal C$, let $C^\circ := C\setminus \cup_{C' \in \mathcal C, C\not\subset C'}  C'$. Then we have partitions
$			C = \sqcup_{C'\in \mathcal C, C'\subset C} C'^\circ
$ and $D=\sqcup_{C\in\mathcal C}C^\circ$. Therefore it suffices to show that the contribution of each $[C^\circ]$ to the right-hand side of the claimed equality is 1. This follows immediately from Lemma \ref{Inclusion-Exclusion for arbitrary set}.
	\end{proof}

	\begin{lemma}\label{lemCls}
		Let $A\subset V$ be a central hyperplane arrangement and let $\mathbb P(V)^{\mathcal S}$ be the canonical log resolution of $\bP(A)$. Then $[\mathbb P(V)^{\mathcal S}]$ is a polynomial in $ \bZ[\mathbb L]$ and
$			[\mathbb P(V)^{\mathcal S}] \equiv 1 (\mathrm{mod}\ \mathbb L).
$	\end{lemma}
	\begin{proof}
		We proceed by induction. The claim is true for $\mathcal S = \emptyset$ and $n = 1$. Take $E = \mu^{-1}(\bP(A)) = \cup_{W \in \mathcal S} E_{W}$ with the notation as in \ref{subsection_canonical_log_resolution}. Then 
$			[\mathbb P(V)^{\mathcal S}] = [\mathbb P(V)]-[\bP(A)]+[E].
$		By inclusion-exclusion  and Lemma \ref{Inclusion-Exclusion for edge sets}, we have
		\begin{displaymath}\textstyle
			[\mathbb P(V)^{\mathcal S}] = [\mathbb P(V)]+\sum_{r = 0}^n\; \sum_{{W_0\subsetneq \ldots \subsetneq W_r, W_i\in \mathcal S}} (-1)^r \left([\cap_{j = 0}^r E_{W_j}]-[\bP(W_0)]\right).
		\end{displaymath}
		By Proposition \ref{propBS}, $\cap_{j = 0}^r E_{W_j}$ is a product of varieties of type $\bP(W_k)^{\mathcal S_k}$ with $\dim W_k < n$ and $\mathcal S_k$ the set of edges of a hyperplane arrangement in $W_k$. Hence $[\cap_{j = 0}^r E_{W_j}]$ is a polynomial in $\mathbb \bZ[\mathbb L]$ and $[\cap_{j = 0}^r E_{W_j}]  \equiv 1 (\mathrm{mod}\ \mathbb L)$ by the induction hypothesis.  
				Since $[\mathbb P(V)] \equiv 1 (\mathrm{mod}\ \mathbb L)$, we are done.
	\end{proof}

\begin{proof}[Proof of Proposition \ref{prop1}]
Recall $\mu : X \to \mathbb P^n=\bP(V)$ is the canonical log resolution of hyperplane arrangement $\mathcal A$, $\mathrm{div}\,(\omega^{1/q}) = \sum_{W\in \mathcal S} (a_W-1)E_W$ as a divisor on $X$, and $a_i := a_{V_i}$ for a hyperplane $V_i \in \mathcal S$. Let $\tilde\omega^{1/q}$ be the multivalued rational $n$-form $\omega^{1/q}$ viewed on $\bP(V)$, so that $\omega^{1/q}=\mu^*\tilde\omega^{1/q}$.
Then $
		\divi( \tilde \omega^{1/q}) = \sum_{i = 1}^d (a_i-1)\;  \mathbb P(V_i)
	$ is given just by the hyperplanes in the arrangement. We have $\divi( \tilde \omega^{1/q}) \sim_\bQ K_{\bP(V)}$ and $\divi(\omega^{1/q})\sim_\bQ K_{X}=K_\mu+\mu^*K_{\bP(V)}$ where $K_\mu$ is the relative canonical divisor. Moreover, after multiplication with $q$ both $\bQ$-linear equivalences become $\bZ$-linear equivalences. By definition of the canonical log resolution, $K_\mu=\sum_{W\in\mathcal S}(\codim \,W -1)E_W$ and $\mu^*K_{\bP(V)}=\sum_{W\in \mathcal S}(\sum_{i:W\subset V_i}(a_i-1))E_W$.
 Hence $\divi(\omega^{1/q})= \sum_{W\in \mathcal S} (b_W-1) E_W$ with $b_W$ as in the statement of the proposition. In particular, $a_W = b_W$ for all $W\in \mathcal S$. Hence $\omega^{1/q}$ has no logarithmic poles if and only if $b_W\neq 0$ for all $W\in\cS$. 

 By Remark \ref{rmkClo}, 
 $$\textstyle\mathbb L^n \cdot \PV\int_{\mathbb P(V)^{\mathcal S}} \omega^{1/q}= [\mathbb P(V)^{\mathcal S}]+\sum_{r \geq 0} \;\sum_{{W_0\subsetneq \ldots \subsetneq W_r, W_i \in \mathcal S}} [\cap_{i =0}^r E_{W_i}] \prod_{i = 0}^r  \left(\frac{\mathbb L-1}{\mathbb L^{b_{W_i}}-1}- 1\right),$$ 
 which lies in $\bZ\llbracket \bL^{1/q}\rrbracket$  by Corollary \ref{corDefRing}. Hence the constant term of this  series  as a formal power series in $\bL^{1/q}$ is well-defined. We denote it by  $\Delta(\mathcal S, \bm a)\in\bZ$. Here $\bm a=(a_1,\ldots, a_d)$. 
	By Proposition \ref{propBS} and Lemma \ref{lemCls},  $[\cap_{i =0}^r E_{W_i}]$ and  $[\mathbb P(V)^{\mathcal S}]$ are  polynomials in $\mathbb L$  with residue class $1$ modulo $\mathbb L$. Hence $\Delta(\mathcal S, \bm a)$ is the constant term of 
		$$\textstyle
			 1+\sum_{r \geq 0}\; \sum_{{W_0\subsetneq \ldots \subsetneq W_r, W_i \in \mathcal S}} \; \prod_{i = 0}^r\left(\frac{\mathbb L-1}{\mathbb L^{b_{W_i}}-1}- 1\right).
		$$
		
		By \eqref{eq*} the leading term of $\frac{\mathbb L-1}{\mathbb L^{u}-1}-1$ is $-1$, $\mathbb L^u$, $0$, $-\mathbb L$ if $u$ is in $(-\infty,0)$, $(0,1)$, $\{1\}$, $(1,\infty)$, respectively. Thus the constant term of $\frac{\mathbb L-1}{\mathbb L^{u}-1}-1$ is  $\lambda(u)$, where $
			\lambda : \mathbb R\setminus\{0\} \to \{-1,0\}$ sends $u$ to 
				$-1$ if  $u < 0$, and to $0$ if $u > 0$.
	Thus we obtain that
$		\Delta(\mathcal S, \bm a) = 1+\sum_{r \geq 0}\; \sum_{{W_0\subsetneq \ldots \subsetneq W_r, W_i \in \mathcal S}} \lambda(b_{W_0})\ldots\lambda(b_{W_r}),
$	
which is equivalent to our claim. 
\end{proof}

In the rest of this subsection, we will produce a tuple $\bm a = (a_1,\ldots,a_d) \in \mathbb (\mathbb Q\setminus \mathbb Z)^d$ such that $\Delta(\mathcal S, \bm a) = 1+\sum_{r \geq 0}\; \sum_{{W_0\subsetneq \ldots \subsetneq W_r, W_i \in \mathcal S}} \lambda(b_{W_0})\ldots\lambda(b_{W_r}) \neq 0$. In particular, for $\tilde \omega^{1/q}$ satisfying $\divi( \tilde \omega^{1/q}) = \sum_{i = 1}^d (a_i-1)\;  \mathbb P(V_i)$ and $\omega^{1/q}=\mu^*\tilde\omega^{1/q}$, the integral $\PV\int_{\mathbb P(V)^{\mathcal S}} \omega^{1/q}$ is non-zero. Indeed, the chosen tuple $\bm a$ will ensure that every $b_W = \codim W+\sum_{W\subset V_i}(a_i-1)$ is positive, and consequently $\lambda(b_W) = 0$ for all $W \in \mathcal S$.

We start with a  characterization of the structure of a central essential indecomposable hyperplane arrangement.

	
	\begin{lemma}\label{a_good_subhyperplane_arrangement}
	Suppose $A$ is essential and indecomposable (so, in particular, $d \geq n+2$), then:
	\begin{enumerate}
	\item After a permutation of indices, we have $V_{i+1} = \{x_i = 0\} \subset A$ for some coordinate system $x_0,\ldots, x_n$ on $\bC^{n+1}$.
	\item Among the rest of the hyperplanes in $A$, there are some hyperplanes $C_1,...,C_r $, $C_i = \{\sum_{j = 0}^n\iota_{ij} x_j = 0\}$, such that, if $O_i := \{j  \in\{0,\ldots,n\}\mid \iota_{ij} \neq 0\}$, then $O_i \cap (\cup_{j=1}^{i-1}O_j) \neq \emptyset$ for $i>1$, and $\cup_{i=1}^r O_i = \{0,\ldots,n\}$. 
	\end{enumerate}
	In particular, for $i=1,\ldots,r$ let $\zeta_i := \vert O_i \setminus \cup_{j=1}^{i-1} O_j \vert$, then $\sum_{i = 1}^r \zeta_i = n+1$.
\end{lemma}
\begin{proof}
	Since $A$ is essential, we may assume $V_1 \cap \ldots \cap V_{n+1} = \{\bm 0\}$ after a permutation of indices. This yields (1). For $1\le i\le d-n-1$ write $V_{n+1+i} = \{\sum_{j = 0}^n z_{ij} x_j = 0\}$ with $z_{ij}\in\bC$, and let $J_i = \{j \in\{0,\ldots,n\}\mid z_{ij} \neq 0\}$. One  picks $C_i$ by the following algorithm:
	
	$\bullet$ Pick $C_1 = V_{n+2}$. Then $O_1 = J_1$. Let $G := \{2,...,d-n-1\}$, $H := J_1$, $l: = 1$.
	
	$\bullet$ Suppose now $l = l_0$. If $H = \{0,...,n\}$, stop. Otherwise pick some $i\in G$ such that $J_i \cap H \neq \emptyset$ and $J_i \setminus H \neq \emptyset$. Update $G$ to be $G\setminus \{i\}$, $H$ to be $H \cup J_i$, $C_{l_0+1} := V_{n+1+i}$, $O_{l_0+1} := J_i$, and $l$ to be $l_0+1$. Run this step again.

	Indeed, the $i \in G$ in the second step always exists since $A$ is indecomposable.  
	Otherwise, the sets $$\textstyle \{V_{i+n+1} \mid i \in G\} \cup \{V_{i+1} \mid i \in \{0,\ldots,n\} \setminus H\} \text{ and } \{V_{i+n+1} \mid i \in \{2,\dots,d-n-1\}\setminus G\} \cup \{V_{i+1} \mid i\in H\}$$ would give a decomposition of $A$ into the zero loci of two polynomials in disjoint sets of variables: \(\{x_j \mid j \in \{0,\dots,n\}\setminus H\}\) and \(\{x_j \mid j \in H\}\), respectively.
\end{proof}
\begin{example}
	Suppose $A = \{x_0x_1x_2x_3(x_0+x_1)(x_2+x_3)(x_1+x_2) = 0\} = \cup_{i=1}^7 V_i \subset \mathbb C^3$, then the steps of the algorithm described above applied to $A$ are recorded below.
	\begin{enumerate}
		\item $l = 1$. Then $C_1 = \{x_0+x_1 = 0\} = 0$, $G = \{2,3\}$, $H = \{0,1\}$, and $\zeta_1 = 2$.
		
		\item $l = 2$. Then one picks $i = 3$ since $J_2 = \{2,3\}$ does not intersect $H$. Data $(G,H,C_2,O_2,\zeta_2)$ are updated to be $(\{2\},\{0,1,2\},\{x_1+x_2 = 0\},\{1,2\},1)$.
		
		\item $l = 3$. Then $i = 2$ is picked and $(G,H,C_3,O_3,\zeta_3)$ are updated to be $(\emptyset,\{0,1,2,3\},\{x_2+x_3 = 0\},\{2,3\},1)$. The algorithm stops here.
	\end{enumerate}
\end{example}

The following lemma is crucial in establishing the positivity of $b_W$.
\begin{lemma}\label{good_numericial_condition}
	With notation  as in Lemma \ref{a_good_subhyperplane_arrangement},
		 for $i=0,\ldots,n$ and $j = 1,\ldots,r$,  let $n_i = \#\{k \in \{1,\ldots,r\}\mid i\in O_k  \}$ and $m_j = \vert O_j\vert$, then $\sum n_i = \sum m_j$.
 Suppose $I \subset \{0,\ldots,n\}$, $J \subset \{1,\ldots,r\}$, $I\cup J \neq \emptyset$. Let $W = (\cap_{i\in I} V_{i+1}) \cap (\cap_{j\in J} C_j)$. Assume $W \neq \{\bm 0\}$ and $0 < \delta \ll 1$. Then
$$\textstyle		\mathrm{codim}\, W - \left(\vert I \vert \cdot \frac{n+1}{n+2} + \sum_{j\in J} \frac{\zeta_j}{n+2} \right) + \delta(\sum_{i\in I} n_i-\sum_{j\in J} m_j) > 0.$$	
\end{lemma}
\begin{proof}  
	Let $e = \mathrm{codim}\, W$, $e' = \vert I\vert+\vert J\vert-e$. Since $\mathrm{codim}(\cap_{i\in I}V_{i+1}) = \vert I \vert$, we have $e' \leq \vert J \vert$. 
	Moreover, $e' = \vert J\vert$ if and only if $e = \vert I\vert$, that is, $\mathrm{span}_{\mathbb C} \{{C}_j^\vee\}_{j\in J} \subset \mathrm{span}_{\mathbb C} \{{V}^\vee_{i+1}\}_{i\in I}$, where ${(\_)^\vee} \subset V^\vee$ means the dual line of a hyperplane. Equivalently, $\bigcup_{j\in J} O_j \subset I$. It suffices to show
$$\textstyle		\frac{\vert I\vert}{n+2} + \vert J\vert-e' + \delta(\sum_{i\in I} n_i-\sum_{j\in J} m_j) > \frac{\sum_{j\in J} \zeta_j}{n+2}.
$$	
If $e' < \vert J\vert$, since $1 > \frac{n+1}{n+2} = \frac{\sum_{j=1}^r \zeta_j}{n+2} \geq \frac{\sum_{j\in J} \zeta_j}{n+2}$ and $0 < \delta \ll 1$, we are done. Otherwise, suppose $e' < \vert J\vert$, then $\sum_{j\in J}\zeta_j \leq \vert\cup_{j\in J} O_j\vert  \leq \vert I\vert$. If $\sum_{j\in J}\zeta_j < \vert I\vert$, we are done since $0 < \delta \ll 1$. Otherwise, beacause $\sum_{j\in J} \zeta_j \leq \vert \cup_{j\in J} O_j\vert$, we conclude that $\cup_{j\in J} O_j = I$. Then it suffices that $\sum_{i\in I} n_i - \sum_{j\in J} m_j > 0$. Note that $m_j$ contributes to each $n_i$ with $i\in O_j \subset I$ by 1. Hence $\sum_{i\in I} n_i - \sum_{j\in J} m_j \geq 0$. Note that $\cup_{j\in J} O_j \subset I \subsetneq \{0,\ldots,n\}$, since $W \neq \{\bm 0\}$. Thus $J \neq \{1,\ldots,r\}$ for $\cup_{j=1}^r O_j = \{0,\ldots,n\}$. Let $j_0 = \min_{j \notin J}\{j\}$. If $j_0 = 1$, then $2\in J$ and $O_2 \cap O_1 \neq \emptyset$. If $j_0 > 1$, $1,\ldots,j_0-1 \in J$ and $(\cup_{j=1}^{j_0-1} O_j) \cap O_{j_0} \neq \emptyset$. In both cases, $O_{j_0}\cap I \supset O_{j_0} \cap (\bigcup_{j\in J} O_j) \subset O_{j_0} \neq \emptyset$. Consequently, $m_{j_0}$ also contributes positively to $\sum_{i\in I} n_i$, which yields the strict inequality.
\end{proof}
\begin{example}
	Again take $A = \{x_0x_1x_2x_3(x_0+x_1)(x_2+x_3)(x_1+x_2) = 0\} = \cup_{i=1}^7 V_i \subset \mathbb C^3$. Let $I = \{0,1,2\}$ and $J = \{1,2\}$. Recall $C_1 = \{x_0+x_1 = 0\} = V_5$ and $C_2 = \{x_1+x_2 = 0\} = V_7$. Then $W = \{x_0=x_1=x_2 = 0\}$ and $(n_0,n_1,n_2,m_1,m_2,\zeta_1,\zeta_2) = (1,2,2,2,2,2,1)$. Then
	\begin{displaymath}\textstyle
		\mathrm{codim}\, W - \left(\vert I \vert \cdot \frac{n+1}{n+2} + \sum_{j\in J} \frac{\zeta_j}{n+2} \right) + \delta(\sum_{i\in I} n_i-\sum_{j\in J} m_j) = 3-( \frac{3\cdot4}{5}+\frac{2+1}{5})+\delta(5-4) = \delta>0.
	\end{displaymath}
\end{example}

\begin{lemma}\label{Genericall_non_vanishing_1}
	Suppose that the central arrangement $A$ is essential and indecomposable. Then there is $\bm a\in(\bQ\setminus \bZ)^d$ satisfying $\sum_{i=1}^da_i=d-n-1$ such that $\Delta(\mathcal S,\bm a) \neq 0$. 
\end{lemma}
\begin{proof} Here $\Delta(\cS,\bm a)$ is the constant term as in the proof of Proposition \ref{prop1}.
	We assume $V_i$ and $C_j$ are as in Lemma \ref{a_good_subhyperplane_arrangement}, and $n_i, m_j$ as in Lemma \ref{good_numericial_condition}. Let $0 < \varepsilon \ll \delta \ll 1$. Let $r' = n+1+r$ and $d' = d-n-1-r$. Define $a_k$ by
	\begin{displaymath}
		a_k - 1 = 
		\begin{cases}
			-\frac{n+1}{n+2} + n_{k-1}\delta-d'\varepsilon, & 1\le k\le n+1,\\
			-\frac{\zeta_j}{n+2}-m_j\delta-d'\varepsilon, & V_k = C_j,\\
			r'\varepsilon, & \mathrm{otherwise}.
		\end{cases}
	\end{displaymath} 
	For an edge $W\in \mathcal S$, let $I_W := \{i\in\{0,\ldots,n\} \mid W \subset V_{i+1}\}$ and $J_W := \{j\in\{1,\ldots, r\} \mid W \subset C_j\}$. Take $M =(d+n+2)^2 4^{\vert \mathcal S\vert}$ a large integer.	If $I_W\cup J_W = \emptyset$, then $b_W \geq 1-M\cdot \varepsilon > 0$. Otherwise, let $W' = (\cap_{i\in I_W} V_{i+1}) \cap (\cap_{j\in J_W} C_j)$, then $\mathrm{codim}\, W' \leq \mathrm{codim}\, W$. We have
$$\textstyle		b_W \geq \mathrm{codim}\, W'  - (\vert I_W \vert \cdot \frac{n+1}{n+2} + \sum_{j\in J_W} \frac{\zeta_j}{n+2} ) + \delta(\sum_{i\in I_W} n_i-\sum_{j\in J_W} m_j) - M\varepsilon.
$$	By Lemma \ref{good_numericial_condition}, taking $\varepsilon \ll \delta$ gives that $b_W > 0$ for all $W \in \mathcal S$. Thus $\Delta(\mathcal S,\bm a) = 1 > 0$.
\end{proof}

\begin{remark} The last three lemmas give a new characterization of the structure of a central essential indecomposable hyperplane arrangement that is of independent interest. This characterization was used as an essential ingredient for the main result of \cite{XY} for example.
\end{remark}

	\subsection{Generic forms}\label{subsection_formalization} In this part we prove Theorem \ref{prop2} and Corollary \ref{propC2}.
	First we develop a formalism to deal with the motivic principal value integrals for hyperplane arrangements.
Let $s_1,\ldots, s_d$ be formal symbols. Define the free abelian group
$$\textstyle		M := (\mathbb Z \oplus \bigoplus_{i=1}^d \mathbb Zs_i) / \mathbb Z\cdot (n+1+\textstyle\sum_{i=1}^d s_i).
$$	By an integer $k\in M$ we mean $k\cdot (1,0,\ldots,0)\in M$ for $k\in\bZ$; this defines an inclusion $\mathbb Z\subset M$.
	Define the ring
	\begin{displaymath}
		R = \mathbb C[\mathbb L^M] := \mathbb C[\mathbb L^{\pm 1},\mathbb L^{\pm s_1},\ldots,\mathbb L^{\pm s_d}]/(\mathbb L^{n+1+s_1+\ldots+s_d}-1).
	\end{displaymath}
	It is isomorphic to
$		\mathbb C[u^{\pm 1},v_1^{\pm 1},\ldots,v_d^{\pm 1}]/(u^{n+1}v_1\ldots v_d-1) \simeq \mathbb C[u^{\pm 1},v_1^{\pm 1},\ldots ,v_{d-1}^{\pm 1}]
$	via $\mathbb L \mapsto u$, $\mathbb L^{s_i} \mapsto v_i$. A monomial in $R$ can be written uniquely as $\mathbb L^m$ for some $m\in M$.

Every element $m\in M$ gives a function  $\bQ^d\cap\{n+1+\textstyle\sum_{i=1}^ds_i=0\}\to \bQ$ by specialization.
		This function 
	 is determined by its specializations $m(\bm u)\in\bQ$ to $d$ arbitrary  $\bQ$-linearly independent vectors $\bm u\in\bQ^d$ such that $\bm u\in\{n+1+\textstyle\sum_{i=1}^ds_i=0\}$.
 To an edge $W \in \mathcal S$ we associate the element $$c_W:=\mathrm{codim}\, W + \textstyle\sum_{W\subset V_i} s_i$$ in $M.$ 
 Let $K :=\cap_{i=1}^d V_i$.
The following is elementary:

	\begin{lemma}\label{coefficients of si pm1}
		For all $W, W'\in \mathcal S$:
		\begin{enumerate} 
				\item There exist $\lambda_1,\ldots,\lambda_d \in \{0,1\}$ and $\lambda_0\in \mathbb Z_{>0}$ such that $c_W = \lambda_0+\sum_{i = 1}^d \lambda_i s_i$. 
		
		\item If $W = K$, $c_W = -\dim K < 0$. If $W \neq K$, then $c_W\in M\setminus\mathbb Z$ and $\{1, c_W\}$ can be extended to a basis of $M$. In particular, $(\mathbb L^{c_{W}}-1)$ is a prime ideal of $R$.
		
		\item  $c_{W}$ and $c_{W'}$ are $\bZ$-linearly dependent if and only if $c_W = \pm c_{W'}$.
		
		\item If  $W\subsetneq W'$, then $c_W$ and $c_{W'}$ are $\mathbb Z$-linearly independent.
		\end{enumerate}
	\end{lemma}

	 In the fraction field of $R$ define
	$$
		\mathcal F_{\mathcal S} 
	:=   \sum_{I\subset \mathcal S} [E_I^\circ] \prod_{W \in I}  \frac{\mathbb L-1}{\mathbb L^{c_{W}}-1} = [(\mathbb P^n)^\mathcal S]+\sum_{r \geq 0} \;\sum_{{W_0\subsetneq \ldots \subsetneq W_r, W_i \in \mathcal S}} [\cap_{j=0}^rE_{W_j}] \prod_{i = 0}^r  \left(\frac{\mathbb L-1}{\mathbb L^{c_{W_i}}-1}- 1\right),
	$$
	since by Proposition \ref{propBS}, in the first sum it suffices to consider only  $I=\{W_1,\ldots , W_{|I|}\}\subset \cS$ with $W_1\subsetneq \ldots\subsetneq W_{|I|}$, and recall that $[(\mathbb P^n)^{\mathcal S}]$ is  $[E_{\emptyset}^\circ]$ in our notation.
	
			\begin{remark}\label{rmkSpP}
		If we specialize $\bm s$ to  $\bm u\in\bQ^d\cap\{n+1+\textstyle\sum_{i=1}^ds_i=0\}$, then $\mathcal F_{\mathcal S}$ is specialized to  $\mathbb L^n\cdot \PV\int_{\mathbb P(V)^{\mathcal S}}\eta$ for $\mathrm{div}\, (\mu_*\eta )= \sum_{i = 1}^d u_i \cdot \mathbb P(V_i)$, assuming that $\eta$ has no logarithmic poles. Recall from the proof of Proposition \ref{prop1} above, that   $\eta$ has no logarithmic poles 
		iff $b_W\neq 0$ in the notation from that proof with $b_W$ depending on $\eta$. Explicitly,  $b_W=c_W(\bm u)=\mathrm{codim}\, W + \sum_{W\subset V_i} u_i$. 
	\end{remark}


	\begin{lemma}\label{non-vanishing of formal PVI is determined generically by PVI} The following are equivalent:
		\begin{enumerate}[(a)]
		\item $\mathcal F_{\mathcal S} \neq 0$ (and, respectively, $\dim K\neq 1$).
		
		\item There is some multivalued rational $n$-form $\eta$ without  logarithmic poles on $\mathbb P(V)^{\mathcal S}$ whose divisor has support in $\cup_{W\in \cS} E_W$ (respectively, exactly $\cup_{W\in \cS} E_W$) such that $\PV\int_{\mathbb P(V)^{\mathcal S}} \eta \neq 0$.
		 
		 \item  	For  generic multivalued rational $n$-forms $\eta$  without  logarithmic poles on $\mathbb P(V)^{\mathcal S}$ whose divisors have support in $\cup_{W\in \cS} E_W$ (respectively, exactly $\cup_{W\in \cS} E_W$),  $\PV\int_{\mathbb P(V)^{\mathcal S}} \eta \neq 0$. Here genericity is  as in Theorem \ref{prop2}.  
		 \end{enumerate}
	\end{lemma}
	\begin{proof} Clearly (c) implies (b).
		For the rest, define first $\mathcal G_{\mathcal S}\in R$ by
\be\label{eqGs}\textstyle
		\mathcal F_{\mathcal S} = {\mathcal G_{\mathcal S}}\cdot{\prod_{W\in \mathcal S}(\mathbb L^{c_{W}}-1)^{-1}}
\ee		
 	So $\mathcal F_{\mathcal S} = 0$ if and only if $\mathcal G_{\mathcal S} = 0$. We can write $\mathcal G_{\mathcal S}$ uniquely as
$			\mathcal G_{\mathcal S} = \sum_{m \in Z} \alpha_m \mathbb L^m$ with $\alpha_m \in \mathbb C,
$		where $Z\subset M$ is a finite subset such that $\alpha_m \neq 0$ for all $m\in Z$.
		
		As in the proof of Proposition \ref{prop1}, the parameter space of multivalued rational $n$-forms  without logarithmic poles on $\mathbb P(V)^{\mathcal S}$ whose divisor  is supported in $\cup_{W\in\cS}E_W$	
		can be  identified with the hyperplane arrangement complement 
		$$
		D_\cS:=(\bQ^d\cap\{n+1+\textstyle\sum_{i=1}^ds_i=0\})\setminus \cup_{W\in\cS}\{c_W(\bm s)=0\},
		$$
		 where $c_W$ are viewed as linear polynomials in $\bm s=(s_1,\ldots, s_d)$. This space is non-empty since $0\not\in\cS$, and thus Zariski open and dense in $\bA_\bQ^{d-1}\cong\{n+1+\textstyle\sum_{i=1}^ds_i=0\}$. If we insist that the forms have support exactly $\cup_{W\in\cS}E_W$, we must  also remove  $\cup_{W\in\cS}\{c_W(\bm s)=1\}$ from $D_\cS$. The resulting space is non-empty if and only if $\dim K\neq 1$.

	Since the specializations of $m\neq m'\in Z$ give two different  functions, there is a dense open subset $U$ of $D_\cS$ such that for all $\bm u\in U$ the values $m(\bm u)$, $m\in Z$, are pairwise different. Namely, $U$ is the complement of $\cup_{m\neq m'\in Z}\{m(\bm s)-m'(\bm s)=0\}$ in $D_\cS$.
		
		  (a) $\Rightarrow$ (c): If $\mathcal G_{\mathcal S} \neq 0$ then  $Z \neq \emptyset$. Since $\mathbb L^{m(\bm u)}$, $m\in Z$, are pairwise different for fixed $\bm u\in U$, we have $\PV\int_{\mathbb P(V)^{\mathcal S}} \omega^{1/q} \neq 0$ if the divisor of $\omega^{1/q}$ on $\bP^n$ is $\sum_{i=1}^d u_i\cdot \bP(V_i)$. 
If in addition $\dim K\neq 1$, we can find such $\bm u$ in $U\setminus\cup_{W\in\cS}\{c_W(\bm s)=1\}$. 
		
		(b) $\Rightarrow$ (a):  $\mathcal F_{\mathcal S} \neq 0$ since $\PV\int_{\mathbb P(V)^{\mathcal S}} \omega^{1/q}$ is a specialization of $\mathcal F_\cS$.
	\end{proof}

\begin{proof}[Proof of Theorem \ref{prop2}.] Note that $X=P(V)^{\mathcal S}$.	By Lemma \ref{Genericall_non_vanishing_1}, the constant term of $\mathbb L^n  \;\PV\int_{\mathbb P(V)^{\mathcal S}} \omega^{1/q}$ is non-zero for some $\omega^{1/q}$ without  logarithmic poles on $\mathbb P(V)^{\mathcal S}$ and with divisor supported in, or even exactly, $\cup_{W\in\cS}E_W$. Hence $\mathbb L^n \; \PV\int_{\mathbb P(V)^{\mathcal S}} \omega^{1/q}\neq 0$, and so the condition (b) of Lemma \ref{non-vanishing of formal PVI is determined generically by PVI} is satisfied. But this is equivalent to the condition (c) of Lemma \ref{non-vanishing of formal PVI is determined generically by PVI}, hence the claim follows.
\end{proof}		

\begin{proof}[Proof of Corollary \ref{propC2}.] 
For the proof we will replace $n$ by $n+1$ so that the notation agrees with the running notation in this section.
Let $A=\cup_{i=1}^df^{-1}_i(0)$ be the associated hyperplane arrangement in $V=\bC^{n+1}$. The blowing up $\ol\mu$ of all edges of $A$ in increasing dimension produces a log resolution of $(V,A)$. The first blow up is that of the origin; here the exceptional divisor is $\bP(V)$ and it intersects the inverse image of $A$ in the projectivized arrangement $\bP(A)$. Moreover, the canonical resolution of $(\bP(V),\bP(A))$ is obtained by base change from $\ol\mu$. Let $g=\prod_{i=1}^df_i^{m_i}$ for $\bm m =(m_1,\ldots, m_d)\in\bZ_{>0}^d$. Consider $R_{\bm 0}$ the ``residue" of the motivic zeta function of $g$ corresponding to the strict transform of the first exceptional divisor of $\ol\mu$, as in Remark \ref{rmkMPIVres}. Note that the numerical genericity condition $\al_W\neq 0$ for all $W\in \cS$ as in Remark \ref{rmkMPIVres} is achieved for $\bm m$ outside the arrangement $\cup_{W\in\cS}\{\nu_W \sum_{i=1}^dm_i-(n+1)\sum_{V_i\supset W}m_i=0\}$. By Remark \ref{rmkMPIVres}, $R_{\bm 0}$ is a motivic principal value integral for a form with divisor $\sum_{i=1}^d(\al_{V_i}-1)\bP(V_i)$ on $\bP(V)$. It follows as in the proof of
Theorem \ref{prop2} that for  $\bm m$ outside another arrangement, this motivic principal value integral is non-zero. Taking $\Theta$ to define the arrangement containing the ``bad" $\bm m$,  we have that
 $-\nu_{\bm 0}/N_{\bm 0}=-(n+1)/\sum_{i=1}^dm_i$ for $\bm m\in\bZ^d_{>0}\setminus\{\Theta=0\}$ is a pole of $Z^{\mathrm{mot}}_{g}(s)$, where $\nu_{\bm 0}/N_{\bm 0}$ is as in Remark \ref{rmkMPIVres} for $D_i$ equal to the strict transform of first exceptional divisor $\bP(V)$,
  since by Remark \ref{rmkMPIVres} the non-vanishing motivic principal value integral implies the non-vanishing of the residue of this pole in this case.
\end{proof}
	
	\subsection{Non-essential and decomposable cases.}\label{subNndd}	
	
	\begin{definition}
		For $W \in \mathcal S$, $c_W$ is called a {\it pole of $\mathcal F_{\mathcal S}$} if $\mathcal G_{\mathcal S}$ is not in the ideal $((\mathbb L^{c_W}-1)^{\kappa_W})$, where $\kappa_W = \#\{W'\in \mathcal S \mid c_{W'} = \pm c_{W}\}$ and $\cG_\cS$ is as in (\ref{eqGs}). 	
		Since $R$ is a UFD and $c_W \neq 0$, we can define the order of $\mathbb L^{c_W}-1$ in the numerator $\cG_\cS$, and in the denominator $\prod_{W\in\cS}(\bL^{c_W}-1)$ of $\cF_\cS$, and a pole occurs when the strict inequality $<$ holds between these two orders.  By Lemma \ref{coefficients of si pm1} we see that this is equivalent to the above definition.
	Consequently, if no $c_W$ is a pole, then $\mathcal F_{\mathcal S} \in R$.
	\end{definition}

	\begin{lemma}\label{FPVI with no pole}
		If there is no pole in $\mathcal F_{\mathcal S}$, then $\mathcal F_{\mathcal S} \in \mathbb C[\mathbb L]$.
	\end{lemma}
	\begin{proof}
		We can write
$			\mathcal F_{\mathcal S} = \sum_{m\in B} \al_m \mathbb L^m$ with $\al_m \in \mathbb C$
		for a finite $B\subset M$. It suffices to show that $m \in \mathbb Z_{\geq 0}$ for all $m\in B$.
		Take $D_{\mathcal S}$ as in the proof of Proposition \ref{non-vanishing of formal PVI is determined generically by PVI}. Suppose that there is some $m\in B\setminus \mathbb Z_{\ge 0}$. Since $D_{\mathcal S}$ is dense in $\bA_\mathbb Q ^{d-1}$, there is some $\bm u \in D_{\mathcal S}$ such that $m(\bm u) < 0$.
		Take $q\in\bZ_{>0}$ such that $q u_i \in \mathbb Z$ for all $1\leq i\leq d$. Then the specialization $\mathcal F_{\mathcal S}(\bm u)$ is $\bL^n$ times the motivic principal value integral of some multivalued $n$-form as in Remark \ref{rmkSpP}. Thus it lies in $\bC\llbracket \bL^{1/q}\rrbracket$ by Corollary	\ref{corDefRing}. This is a contradiction, since we have a term $\mathbb L^{m(\bm u)}$ with negative degree in $\cF_\cS(\bm u)$.
	\end{proof}

	\begin{prop}\label{non-essential PVI vanishes}
		If a central hyperplane arrangement $A\subset V$ is not  essential, then $\mathcal F_{\mathcal S} = 0$.
	\end{prop}
	\begin{proof} By assumption, $k:=\dim K\neq 0$, where recall that $K=\cap_{i=1}^dV_i$. Then
$$		\textstyle	\mathcal F_{\mathcal S} = \sum_{I \subsetneq \mathcal S\setminus \{K\}} \left([E_I^\circ]+[E_{I\cup \{K\}}^\circ]\frac{\mathbb L-1}{\mathbb L^{c_K}-1}\right) \prod_{W\in I} \frac{\mathbb L-1}{\mathbb L^{c_W}-1}.
$$		It suffices to prove that the term
$			[E_I^\circ]+[E_{I\cup \{K\}}^\circ](\mathbb L-1)(\mathbb L^{c_K}-1)^{-1} = 0.
$		Suppose $I=\{W_i\mid 1\le 1\le r\}$ with $K\neq W_1\subsetneq W_2 \subsetneq \ldots \subsetneq W_r$, and take $W_0 =  \bm 0 $. Then by Corollary \ref{E_I_circ forumla},
$$	\textstyle		[E_I^\circ]  = [U(\mathcal S_{ \bm 0 }^{W_1})]\cdot \prod_{i = 2}^{r+1} [U(\mathcal S_{W_{i-1}}^{W_i})],
\quad		[E_{I\cup \{K\}}^\circ]  = [U(\mathcal S_{ \bm 0 }^K)]\cdot [U(\mathcal S_{K}^{W_1})]\cdot \prod_{i = 2}^{r+1} [U(\mathcal S_{W_{i-1}}^{W_i})].
$$		For $W \in \mathcal S\cup\{V\}$, we have
$			U(\mathcal S_{ \bm 0 }^{W}) = U(\mathcal S_{K}^W) \times \mathbb C^{k}.
$		We also have $c_K = -k\in M$, and $U(\mathcal S_{ \bm 0 }^K) = \mathbb P^{\,k-1}$. So it suffices to prove that
		$
			\mathbb L^{k} + [\mathbb P^{\,k-1}](\mathbb L-1)(\mathbb L^{-k}-1)^{-1} = 0,
		$
which is easy.
	\end{proof}
	
	By Lemma \ref{non-vanishing of formal PVI is determined generically by PVI}, this implies:
	
	\begin{cor}
	If a central hyperplane arrangement $A\subset V$ is not  essential, then $\PV\int_{\bP(V)^\cS}\omega^{1/q}=0$ for every multivalued rational $n$-form $\omega^{1/q}$ without logarithmic poles whose divisor has support contained in the inverse image of $\bP(A)$.
	\end{cor}

 We  introduce some additional notation. As before, $\mathcal S$ is the set of all edges of $A\subset V$ different than $\bm 0$, and $\cS^{W}_{W'}$ is the set of edges different than the origin of the associated arrangement in $W/W'$ for $W'\subsetneq W$ with $W', W\in \cS\cup\{\bm 0, V\}$. For $W''\in I\subset\cS^{W}_{W'}$ we define $_{W/W'}  E_{W''}$ to be the prime divisor on the canonical log resolution of $(\mathbb P(W/W'), \cS^W_{W'})$ corresponding to $W''$. Let
\be\label{eqW0}
_{W/W'} E_I:=\cap_{W''\in I} \;\;{_{W/W'}} E_{W''},   \quad\quad  _{W/W'} E_I^\circ:={_{W/W'}}E_I\setminus \cup_{W''\in \cS^{W}_{W'}\setminus I} \;\;{_{W/W'}} E_{W''}.
\ee

	\begin{lemma}\label{exceptioanl divisor circ cut by V}
		With notation of subsection \ref{subsection_canonical_log_resolution}, if $W \in \mathcal S$ and $I\subset \mathcal S^{W}$, then 
$			[E_{I\cup \{W\}}^\circ] = [{_{W/\bm 0}  E_I^\circ}]  [U(\mathcal S^{V}_{ W})]$ in $K_0(Var_\bC)$, with $_{W/\bm 0}  E_I^\circ$ as in (\ref{eqW0}).
\end{lemma}
\begin{proof}
		It is enough to consider a chain $I = \{W_1 \subsetneq \ldots \subsetneq W_r\}$ where $W_r\subsetneq  W$. Let  $W_0=\bm 0$, $W_{r+1} =  W$. 
				Then by Corollary \ref{E_I_circ forumla}, we have
$		[E_{I\cup \{W\}}^\circ]  = \prod_{i = 1}^{r+1} [U(\mathcal S_{W_{i-1}}^{W_i})]  \cdot [U(\mathcal S^{V}_{W_{r+1}})]$ and $ [{_{V'\times W} E_I^\circ}] = \prod_{i = 1}^{r+1} [U ((\cS^{V'\times W})_{W_{i-1}}^{W_i} )].
$
By definition, the pairs $(W_i/W_{i-1}, \cS_{W_{i-1}}^{W_i})$ and $(W_i/W_{i-1},U((\cS^{V'\times W})_{W_{i-1}}^{W_i}))$ define the same hyperplane arrangement. Hence, $[U(\mathcal S_{W_{i-1}}^{W_i})] = [U((\cS^{V'\times W})_{W_{i-1}}^{W_i})]$.
	\end{proof}

		We suppose now that $A\subset V$ is a central decomposable arrangement. So, we can assume that there exist $n', n''\in\bZ_{>0}$ with $n=n'+n''$, and there exist non-constant reduced homogeneous polynomials $f'\in \mathbb C[x_1,\ldots,x_{n'}]$ and $f''\in \mathbb C[x_{n'+1},\ldots,x_{n}]$, such that $A = \{f'f'' = 0\}$. Set $V'=\bC^{n'}$, $V''=\bC^{n''}$, so that $V=V'\times V''$.
		 $\cT$ be the set of  edges of $\{f'' = 0\} \subset  V''$, possibly equal to the origin in $V''$, together with $V''$.

	\begin{lemma}\label{Vanishing for concentrated sum}  If  $W \in \mathcal T$ then
		$T(V,\cS, V'\times W):=\sum_{I\subset \mathcal S^{V'\times W}\cup\{V'\times W\}} [E_I^\circ] (1-\mathbb L)^{\vert I \vert} = 0 \in \mathbb C[\mathbb L].$ One has also an obvious counterpart by switching the roles of $V'$ and $V''$.

	\end{lemma}
	\begin{proof} Since increasing chains  in $\mathcal S^{V'\times W}\cup\{V'\times W\}$ can be divided into two types, having $V'\times W$ as the tail or not, it suffices to prove  for an increasing chain $ I = \{W_1 \subsetneq \ldots \subsetneq W_r \}$ in $\cS^{V'\times W}$ that 
$			[E_I^\circ] + [E_{I\cup \{V'\times W\}}^\circ](1-\mathbb L) = 0.
$		By Corollary \ref{E_I_circ forumla} we see that it suffices to prove that $[U(\mathcal S_{W_r}^V)] = [U(\mathcal S_{W_r}^{V'\times W})]  [U(\mathcal S_{V'\times W}^{V})]  (\mathbb L-1)$. 
Let $M$, $M_1$, $M_2$ be the complements in $V/W_r$, $(V'\times W)/W_r$, $V/(V'\times W)$, respectively, of the  hyperplane arrangements induced by $A$. Since the complement of an affine central hyperplane arrangement is a $\bC^*$-fibration over the complement of the induced projective arrangement, and since $\bL-1$ is not a zero divisor in $\bC[\bL]$, we see that it is enough to show that $[M]=[M_1][M_2]$. Since $A$ is decomposable, $W_r=W'\times W$ for some edge $W'$ of $\{f_1=0\}\subset V_1$.  Then $M=M_1\times M_2$, from which the claim follows.
	\end{proof}

		\begin{lemma}\label{Decomposable constant FPVI is zero}
		If $\mathcal F_{\mathcal S} \in \mathbb C[\mathbb L]$, then $\mathcal F_{\mathcal S} = 0$. 
	\end{lemma}
	\begin{proof}
		Suppose $\mathcal F_{\mathcal S} = \sum_{i = 0}^r \al_i \mathbb L^i\in \mathbb C[\mathbb L] $, where $\al_i \in \mathbb C$ and $r \in \mathbb N$. 
	Consider the specialization $\bm s\to \bm u$ from Remark \ref{rmkSpP}, 		giving rise to the specialization $\mathcal F_{\mathcal S}\to \mathbb L^n\cdot \PV\int_{\mathbb P(V)^{\mathcal S}}\eta$ for $\mathrm{div}\, (\mu_*\eta )= \sum_{i = 1}^d u_i \cdot \mathbb P(V_i)$ when $\bm u$ is generic enough so that $\eta$ has no logarithmic poles. In the case of our lemma, $\cF_\cS$ has constant specializations since $\cF_\cS$ is a constant in the variables $s_i$.

	Consider the specialization $s_i$ to $u_i \in \mathbb Z$  such that $u_1 + \ldots  +u_d = -n-1$ and with  $u_2,\ldots ,u_d \ll 0$ very negative. So $u_1 \gg 0$ is very positive. We can assume that $u_1\gg -(u_2+\ldots+u_{d-1})\gg 0$, that is, $u_1$ is very large compared with $-(u_2+\ldots+u_{d-1})$ which is also sufficiently large.
	Note that indeed this avoids logarithmic poles as in Remark \ref{rmkSpP}. 
	To make the bounds on the $u_i$ precise, let $\theta > n+r+2$ be a large positive integer and we take  $u_2,\ldots ,u_d < -\theta$. We will use that for any subset $\Lambda\subsetneq \{2,\ldots,d\}$, we have $u_1 = -n-1-(u_2+\ldots+u_d) > -n-1-\sum_{j\in \Lambda} u_j + (d-1-\vert \Lambda\vert) \theta > -\sum_{j\in \Lambda} u_j+r+1$.
	
Let $K$ be the smallest edge of $A$. Since $A$ is decomposable, $K=K'\times K''$ for some subspaces $K'\subset V'$ and $K'' \subset V''$.

		For $W \in \mathcal S$, recall that $b_W = \mathrm{codim}\, W + \sum_{W\subset V_i} u_i$  and $\bm 0\not\in \mathcal S$. 
If $K =  \bm 0$, then $b_W \gg 0$ if $W \subset V_1$, and $b_W \ll 0$ if $W \not\subset V_1$. The bounds on $b_W$ can be made effective in terms of the previously chosen bounds on the $u_i$. Explicitly, if $K =  \bm 0$, then $b_W > \codim W+r+1 > 0$ if $W \subset V_1$, and $b_W < \codim W - \sum_{W \subset V_i} \theta < -r-1 < 0$ if $W \not\subset V_1$.
		
	Expand the specialization of $\cF_\cS$ to $\bm u$ into a power series. By \eqref{eq*}, its residue class in $\mathbb Z[[L]]/(\mathbb L^{r+1})$ is
$			\sum_{i = 0}^r \al_i \mathbb L^i = \sum_{I \subset \mathcal S^{V_1}\cup\{V_1\}} [E_I^\circ] (1-\mathbb L)^{\vert I\vert}=0
$		by Lemma \ref{Vanishing for concentrated sum}. If $K \neq  \bm 0 $, then $b_W > \codim W + r+ 1$ if $K \neq W \subset V_1$,  $b_W < -r-1$ if $W \not\subset V_1$, and   $b_K=-\dim K<0$. Divide the chains in $\mathcal S$ into two types, having $K$ as the head or not. Again, modulo $\mathbb L^{r+1}$, by \eqref{eq*} we have 
$$\textstyle			\sum_{i = 0}^r \al_i \mathbb L^i = \sum_{K \notin I \subset \mathcal S^{V_1}\cup\{V_1\}} (1-\mathbb L)^{\vert I\vert}\left([E_I^\circ]  + [E_{I\cup\{K\}}^\circ] (\bL-1)(\bL^{\dim K}+\bL^{2\dim K}+\ldots) \right).$$
 Denote by  $p':V=V'\times V''\to V'$ the first projection. We have that 
$$
\textstyle
\sum_{K \notin I \subset \mathcal S^{V_1}\cup\{V_1\}} (1-\mathbb L)^{\vert I\vert}[E_I^\circ] =
 T(V, \cS, p'(V_1)\times V'') - [K]\cdot T(V/K,\cS^V_K,p'(V_1)/K'\times V''/K'') =0$$
by Lemma \ref{Vanishing for concentrated sum}. 
So 
$			\sum_{i = 0}^r \al_i \mathbb L^i = - (1+\bL^{\dim K}+\bL^{2\dim K}+\ldots)\sum_{K\in J\subset\cS^{V_1}\cup\{V_1\}}[E_J^\circ](1-\bL)^{|J|},$
	but the sum over $J$ here equals
			$ T(V,\cS,p'(V_1)\times V'') - \sum_{K\not\in J\subset\cS^{V_1}\cup\{V_1\}}[E_J^\circ](1-\bL)^{|J|} =0.
$	\end{proof}

	 	\begin{prop}\label{Vanishing_in_decomposable_case} If a central hyperplane arrangement $A\subset V$ is decomposable, then $\cF_\cS=0$.
	\end{prop}
\begin{proof}
		We induct on the dimension of $V$.  		
		This is easy to check for $n = \dim V=2$, when $A$ can only be the union of two lines and the vanishing follows from Theorem \ref{prop3} and Lemma \ref{non-vanishing of formal PVI is determined generically by PVI}. The case $n=3$ is covered via Lemma \ref{non-vanishing of formal PVI is determined generically by PVI} by \cite[0.4]{Vanishing_of_Principal_Value_Integrals_on_Surfaces} mentioned in the introduction, since decomposability is equivalent to $\chi(\bP(V)\setminus\bP(A))=0$. We address now higher dimensions.

			If $\mathcal F_{\mathcal S}$  has no poles, the claim follows by Lemma \ref{FPVI with no pole} and Lemma \ref{Decomposable constant FPVI is zero}. So it suffices to prove $\mathcal F_{\mathcal S}$ has no poles.   By Theorem \ref{non-essential PVI vanishes}, we may also assume that $A$ is essential.
		
%
%
%
%
%
		
		Suppose $c_{Z}$ with $Z \in \mathcal S$ is a pole. Then it is contributed by those chains in $\cS$ containing some $W$ such that $c_{W} = \pm c_Z$. By Lemma \ref{coefficients of si pm1}, no more than one $W$  in a chain can have $c_W=\pm c_Z$. Therefore, the sum of all terms in $\cF_\cS$ that may contribute to the pole $c_Z$ is
\be\label{eqWW}		\textstyle	\sum_{W\in\cS: c_W = \pm c_Z}\sum_{W \in I \subset \mathcal S} [E_I^\circ]  \prod_{W'\in I} \frac{\mathbb L-1}{\mathbb L^{c_{W'}}-1}.
\ee

For a chain $I = \{W_1\subsetneq \ldots  \subsetneq  W_r \subsetneq W \subsetneq W'_1 \subsetneq \ldots  \subsetneq W'_{r'}\}$ in $\cS$ we have
\be\label{eqPro}
\textstyle			[E_I^\circ] = \left(\prod_{i = 1}^{r+1} U\left(\mathcal S_{W_{i-1}}^{W_i}\right)\right)  \left(\prod_{i = 1}^{r'+1} U\left(\mathcal S_{W_{i-1}'}^{W_i'}\right)\right),
\ee	where $W_0 = \{\bm 0\}, W_{r+1} = W'_{0} = W$, and $W'_{r'+1} = V$, by Corollary \ref{E_I_circ forumla}.  Since $J = \{W_1\subsetneq \ldots  \subsetneq W_r\}$ and $H = \{W'_1 \subsetneq \ldots  \subsetneq W'_{r'}\}$ run through all increasing chains in $\mathcal S^W$ and $\mathcal S_W$ respectively,  $\sum_{W \in I \subset \mathcal S} [E_I^\circ]  \prod_{W'\in I} \frac{\mathbb L-1}{\mathbb L^{c_{W'}}-1}$ can be split as as follows. By Corollary \ref{E_I_circ forumla} and \eqref{eqPro}, we have
\begin{displaymath}\textstyle
	[_{W/\bm 0} E_J^\circ] \cdot [_{V/W}E_H^\circ] = \left(\prod_{i = 1}^{r+1} U\left(\mathcal S_{W_{i-1}}^{W_i}\right)\right)  \cdot \left(\prod_{i = 1}^{r'+1} U\left(\mathcal S_{W_{i-1}'}^{W_i'}\right)\right) = [E_{J \cup H \cup \{W\}}^\circ].
\end{displaymath}
Hence,
\begin{align*}
	& \textstyle \sum_{W \in I \subset \mathcal S} [E_I^\circ]  \prod_{W'\in I} \frac{\mathbb L-1}{\mathbb L^{c_{W'}}-1} \\
	& \textstyle= \sum_{J \subset \mathcal S^W, H \subset \mathcal S_W} [E_{J \cup H \cup \{W\}}^\circ] \prod_{W'\in J \cup H \cup \{W\}} \frac{\mathbb L-1}{\mathbb L^{c_W'}-1} \\
	& \textstyle= \sum_{J \subset \mathcal S^W, H \subset \mathcal S_W} [_{W/\bm 0} E_J^\circ] \cdot [_{V/W}E_I^\circ] \prod_{W'\in J \cup H \cup \{W\}} \frac{\mathbb L-1}{\mathbb L^{c_W'}-1} \\
	& \textstyle = \frac{\mathbb L-1}{\mathbb L^{c_W}-1}\left (\sum_{{J}\subset \mathcal S^W} [_{W/\bm 0}E_{{J}}^\circ] \prod_{W'\in {J}} \frac{\mathbb L-1}{\mathbb L^{c_{W'}}-1}\right) \left (\sum_{{H}\subset \mathcal S_W} [_{V/W}E_{{H}}^\circ] \prod_{W'\in {H}} \frac{\mathbb L-1}{\mathbb L^{c_{W'}}-1}\right) .
\end{align*}


	For $(W,\mathcal S^W)$ with $\cS^W\neq\emptyset$ and for $(V/W,\mathcal S_W^{V})$ with $\cS^V_W\neq\emptyset$, consider their formal motivic principal value integrals $\mathcal F_{\mathcal S^W} \in \mathrm{Frac}(\mathbb C[\mathbb L^{M^W}])$ and $\mathcal F_{\mathcal S_W^V} \in \mathrm{Frac}(\mathbb C[\mathbb L^{M_W}])$. Here 
$$ \textstyle M^W=\left(\bZ\oplus\bigoplus_{W'\in\cS^W:\, \codim_WW'=1}\bZ \hat  s_{W'}\right)/\bZ\left(\dim W+\sum_{W'\in\cS^W:\,\codim_WW'=1}\hat s_{W'}\right),$$
$$ \textstyle
 M_W=\left(\bZ\oplus\bigoplus_{i:W\subsetneq V_i}\bZ\bar s_i\right)/\bZ\left(\dim V/W +\sum_{i:W\subsetneq V_i}\bar s_i\right),
$$
and $1, \hat s_{W'}, \bar s_i$ are formal symbols.

	There is a monomomorphism
$		 M^W \hookrightarrow M/\mathbb Z c_W$ defined by $\hat{s}_{W'} \mapsto \sum_{i: V_i\cap W= W'} s_i$ and  $1\mapsto 1.
$ It gives a monomorphism
$			\mathbb C[\mathbb L^{M^W}] \hookrightarrow \mathbb C[\mathbb L^M]/(\mathbb L^{c_W}-1).
$	
We note that for $W''\in \cS^W$, the element $\hat c_{W''}:=\codim_WW''+\sum_{W': W''\subset W', \,\codim_WW'=1}\hat s_{W'}\in M^W$ from the definition of $\cF_{\cS^W}$, is mapped to the class of $c_{W''}$ in $M/\bZ c_W$.

We also have a monomorphism
		$ M_W \hookrightarrow M/(\mathbb Z c_W)$ defined by  $ \bar{s}_{i} \mapsto s_{i}$ and $1\mapsto 1,
$		and  a corresponding ring monomorphism
$			\mathbb C[\mathbb L^{M_W}] \hookrightarrow \mathbb C[\mathbb L^M]/(\mathbb L^{c_W}-1).
$		
For $W'\in \cS_W$, the image of the element $\ol c_{W'}:=\codim_{V/W}(W'/W) +\sum_{i: W'\subset V_i}\ol s_i\in M_W$ is the class of $c_{W'}$ in $M/\bZ c_{W'}$.

Denote the images of $\mathcal F_{\mathcal S^W}$ and $\mathcal F_{\mathcal S_W^V}$ in $\mathbb C[\mathbb L^M]/(\mathbb L^{c_W}-1)$ by $\mathcal F^W$ and $\mathcal F_W$, respectively.
Then
$$\textstyle			\sum_{I\subset \mathcal S^W} [_{W/\bm 0}E_{I}^\circ] \prod_{W'\in I} \frac{\mathbb L-1}{\mathbb L^{c_{W'}}-1} \equiv \mathcal F^W\ \mathrm{modulo}\ (\mathbb L^{c_W}-1).
$$ Here, both sides are considered in the quotient of the localization of $\bC[\bL^{M}]$  with respect to $\bL^{c_{W'}}-1$ for $W'\in\cS^W$ by the ideal generated by $\mathbb L^{c_W}-1$. This is the same as taking the quotient modulo $\mathbb L^{c_W}-1$ first and then localizing, since each $\mathbb L^{c_{W'}}-1$ is coprime to $\mathbb L^{c_W}-1$ by Lemma \ref{coefficients of si pm1} . 
 Similarly,
$$	\textstyle		\sum_{I\subset \mathcal S_W} [ _{V/W}E_{I}^\circ] \prod_{W'\in I} \frac{\mathbb L-1}{\mathbb L^{c_{W'}}-1} \equiv \mathcal F_W\ \mathrm{modulo}\, (\mathbb L^{c_W}-1).
$$		Hence,
$			\sum_{W \in I \subset \mathcal S} [E_I^\circ]  \prod_{W'\in I} \frac{\mathbb L-1}{\mathbb L^{c_{W'}}-1}$ is equal to $(\bL-1)/(\bL^{c_W}-1)$ times a term that is congruent to 
$\mathcal F_W\mathcal F^W$ modulo $(\bL^{c_W}-1)$.		
		
	If $Z \not\in\{  \bm 0  \times V'', V'\times  \bm 0\}$, then the same holds for $W$ with $c_W = \pm c_Z$. 
	Then one of $(W,\mathcal S^W)$ and $(V/W,\mathcal S_W^{V})$ must come from a decomposable arrangement. By the induction hypothesis, $\mathcal F_W\mathcal F^W=0$. Hence $\sum_{W \in I \subset \mathcal S} [E_I^\circ] \cdot \prod_{W'\in I} \frac{\mathbb L-1}{\mathbb L^{W'}-1}$ is in the ideal $(\bL^{c_W}-1)$ and  it does not contribute  to the pole $c_Z$.

		If $Z =  \bm 0 \times V''=:W_1$ or $Z=V'\times  \bm 0 =:W_2$. Then $c_{W_1}+c_{W_2} = n+1+\sum_{j = 0}^d s_j \equiv 0\in M$. Note that $\mathcal F_{W_1} = \mathcal F^{W_2} =: \mathcal F_1$ and $\mathcal F^{W_1} = \mathcal F_{W_2} =: \mathcal F_2$. Only the terms of (\ref{eqWW})  with $W=W_1, W_2$ can possibly contribute to the pole $c_Z = \pm c_{W_1} = \pm c_{W_2}$. These two terms are $\frac{\bL-1}{\bL^{c_{W_1}}-1}\cF_1\cF_2$ and $\frac{(1-\bL)\bL^{c_{W_1}}}{\bL^{c_{W_1}}-1}\cF_1\cF_2$, respectively. Hence again by induction we see that $c_Z$ is not a pole. 
	\end{proof}

	\begin{proof}[Proof of Theorem \ref{thm1}.] 
	Assume first that the support of $\divi(\omega^{1/q})$ is the whole inverse image of $\cA$ and let $U = X\setminus \mu^{-1}(\mathcal A)$.
	If the cone $A$ of $\cA\subset \bP^n$ is indecomposable, then $\chi(U) \neq 0$. Since this is also the Euler characteristic of all rank-one local systems on $U$, see \cite[Proposition 2.5.4]{Di}, 		
		the conjecture is trivially true in this case. If $A$ is decomposable, then $\cF_\cS=0$ by
	Proposition \ref{Vanishing_in_decomposable_case}. Hence in this case  $\PV \int_{X}\omega^{1/q}= 0$,
by	Lemma	\ref{non-vanishing of formal PVI is determined generically by PVI}.

Suppose now only that the support of $\divi(\omega^{1/q})$ contains $E_W$ for every dense edge $W$ of $A$. Then, by the additivity of topological Euler characteristic on algebraic varieties applied to $X\setminus|\divi(\omega^{1/w})|=U\cup (\cup_{I\cap \cS'=\emptyset}E_I^
\circ)$, we have
$$\textstyle
\chi(X\setminus|\divi(\omega^{1/w})|) =\chi(U)+\sum_{I\cap \cS'=\emptyset}\chi( E_I^\circ) = \chi(U) + \sum_{I\subset \cS''}\chi(E_I^\circ),
$$
where $\cS'$ is the subset of $W\in \cS$ such that $W\in |\divi(\omega^{1/w})|$ and $\cS''=\cS\setminus\cS'$. On the other hand, if $I\subset \cS''$ is a chain $W_1\subset\ldots\subset W_r\in\cS''$, since $W_r$ in non-dense we have $\chi(U(\cS^V_{W_r}))=0$, and hence also $\chi(E_I^\circ)=0$ by Corollary \ref{E_I_circ forumla}. Thus $\chi(X\setminus|\divi(\omega^{1/w})|) =\chi(U)$ and the proof now is exactly as in the previous case.
	\end{proof}

\begin{proof}[Proof of Theorem \ref{thrmConj2}] If the canonical log resolution of $(\bP^n,\bP(A))$ is a good log resolution with respect to $\bm a$, then the proof is the same as that of Theorem \ref{thm1}. There might be the case though that the canonical log resolution is a not a good log resolution with respect to $\bm a$. Note that our definition of $\cF_\cS$ is a particular case, for the canonical log resolution, of a multi-variable analog of the motivic zeta function of \cite[1.6 (i)]{Motivic_Principal_Value_Integrals_Veys}. Namely, for a log resolution $h:Y\to \bP^n$ of $\cA=\bP(A)=\cup_{i=1}^d\bP(V_i)$, let $S$ be the index set of the irreducible components $E_j$ of the union of $h^{-1}(\cA)$ with the exceptional locus, let $\nu_j-1$ be the multiplicity of $E_j$ in $K_{Y/ \bP^n}$, and $N_{ij}$ be the multiplicity of $E_j$ in $h^*(\bP(V_i))$. Define 
$$\textstyle
Z_{\bP^n}(\cA;s_1,\ldots, s_d):=\sum_{J\subset S}[E_I^\circ]\prod_{j\in J}\frac{\bL-1}{\bL^{\nu_j+\sum_{i=1}^dN_{ij}s_i}-1}.
$$
As in \cite{Motivic_Principal_Value_Integrals_Veys}, one can use the weak factorization theorem to show that $Z_{\bP^n}(\cA;s_1,\ldots, s_d)$ is independent of the choice of log resolution. Moreover, $\bL^n\PV(\bP^n,\bm a)$ is a specialization of $Z_{\bP^n}(\cA;s_1,\ldots, s_d)$ when $h$ is a good log resolution with respect to $\bm a$.
Now, $\cF_\cS$ is  the above formula for $Z_{\bP^n}(\cA;s_1,\ldots, s_d)$ when $h$ is the canonical log resolution, viewed in $\text{Frac}(R)$ as in \ref{subsection_formalization}. In the only case we need to consider, the case when $A$ is decomposable, $\cF_\cS=0$ by Proposition \ref{Vanishing_in_decomposable_case}. Hence $\PV(\bP^n,\bm a)=0$.
\end{proof}


\begin{thebibliography}{BMT11}

\bibitem[BMT11]{Monodromy_Conjecture_for_Hyperplane_Arrangement_Budur_Mustata_Teitler}
N.~Budur, M.~Musta\c{t}\u{a}, Z.~Teitler.
\newblock The monodromy conjecture for hyperplane arrangements.
\newblock {\em Geom. Dedicata}, 153: 131--137, 2011.

\bibitem[BS10]{Jumping_Coefficient_and_Spectrum_of_a_Hyperplane_Arrangement_Budur_Saito}
N.~Budur, M.~Saito.
\newblock Jumping coefficients and spectrum of a hyperplane arrangement.
\newblock {\em Math. Ann.}, 347: 545--579, 2010.

\bibitem[BSY11]{Local_Zeta_Function_And_b_Function_of_Certain_Hyperplane_Arragement_Budur_Saito_Sergey}
N.~Budur, M.~Saito,  S.~Yuzvinsky.
\newblock On the local zeta functions and the {$b$}-functions of certain
  hyperplane arrangements.
\newblock {\em J. Lond. Math. Soc.}, 84: 631--648, 2011.
\newblock With an appendix by Willem Veys.



\bibitem[Bud09]{B-uls}
N.~Budur.
\newblock Unitary local systems, multiplier ideals, and polynomial periodicity
  of {H}odge numbers.
\newblock {\em Adv. Math.}, 221: 217--250, 2009.

\bibitem[BW20]{BWabs}
N.~Budur, B.~Wang.
\newblock Absolute sets and the decomposition theorem.
\newblock {\em Ann. Sci. \'Ec. Norm. Sup\'er.}, 53: 469--536, 2020.


\bibitem[Den91]{Distribution_3}
J.~Denef.
\newblock Report on {I}gusa's local zeta function.
\newblock No. 201-203, 359--386.
\newblock {\em S\'eminaire Bourbaki}, Vol.\ 1990/91.

\bibitem[DJ98]{On_the_Vanishing_of_Principal_Value_Integrals_Denef_Jacobs}
J.~Denef, P.~Jacobs.
\newblock On the vanishing of principal value integrals.
\newblock {\em C. R. Acad. Sci. Paris S\'er. I Math.}, 326: 1041--1046, 1998.



\bibitem[DL98]{Motivic_Zeta_Function}
J.~Denef, F. ~Loeser.
\newblock Motivic {I}gusa zeta functions.
\newblock {\em J. Algebraic Geom.}, 7: 505--537, 1998.


\bibitem[Di04]{Di} A. Dimca. {\it Sheaves in topology.}  Springer-Verlag, Berlin, 2004. xvi+236 pp.

\bibitem[ESV92]{ESV}
H.~Esnault, V.~Schechtman, E.~Viehweg.
\newblock Cohomology of local systems on the complement of hyperplanes.
\newblock {\em Invent. Math.}, 109: 557--561, 1992.


\bibitem[Jac00a]{Distribution_1}
P.~Jacobs.
\newblock The distribution {$|f|^\lambda$}, oscillating integrals and principal
  value integrals.
\newblock {\em J. Anal. Math.}, 81: 343--372, 2000.

\bibitem[Jac00b]{Distribution_2}
P.~Jacobs.
\newblock Real principal value integrals.
\newblock {\em Monatsh. Math.}, 130: 261--280, 2000.

\bibitem[Lan83]{Langlands1}
R.~P. Langlands.
\newblock Orbital integrals on forms of {${\rm SL}(3)$}. {I}.
\newblock {\em Amer. J. Math.}, 105: 465--506, 1983.

\bibitem[LS89]{Langlands2}
R.~P. Langlands, D.~Shelstad.
\newblock Orbital integrals on forms of {${\rm SL}(3)$}. {II}.
\newblock {\em Canad. J. Math.}, 41: 480--507, 1989.

\bibitem[STV95]{Resolution_of_Hyperplane_Arrangements_Schchtman_Terao_Varchenko}
V.~Schechtman, H.~Terao,  A.~Varchenko.
\newblock Local systems over complements of hyperplanes and the {K}ac-{K}azhdan
  conditions for singular vectors.
\newblock {\em J. Pure Appl. Algebra}, 100: 93--102, 1995.

\bibitem[SZ24]{SZ24b}
Q.~Shi, H.~Zuo.
\newblock Variation of archimedean zeta function and $n/d$-conjecture for generic multiplicities.
\newblock {\em arXiv:2411.00757.}

\bibitem[Vey93]{VeyR} 
W.~Veys.
\newblock Poles of Igusa's local zeta function and monodromy.
\newblock {\em Bull. Soc. math. France}, 121: 545--598, 1993. 

\bibitem[Vey06]{Vanishing_of_Principal_Value_Integrals_on_Surfaces}
W.~Veys.
\newblock Vanishing of principal value integrals on surfaces.
\newblock {\em J. Reine Angew. Math.}, 598: 139--158, 2006.

\bibitem[Vey07]{Motivic_Principal_Value_Integrals_Veys}
W.~Veys.
\newblock On motivic principal value integrals.
\newblock {\em Math. Proc. Camb. Philos. Soc.}, 143: 543--555, 2007.

\bibitem[XY25]{XY} B. Xie, C. Yu. The $n/d$-Conjecture for nonresonant hyperplane arrangements. arXiv:2501.05189.



\end{thebibliography}
\end{document}